\documentclass[10pt]{amsart}
\usepackage[utf8]{inputenc}
\usepackage{amsmath,amsfonts,amssymb,amsthm,amscd,amsbsy,epsfig,array,mathrsfs,tikz,tikz-cd,url}
\usetikzlibrary{positioning}
\usetikzlibrary{decorations.pathmorphing}
\usetikzlibrary{decorations.markings}

\newcommand{\set}[1]{ \{1,\ldots,   {#1} \} }
\newcommand{\m}[1]{\mathcal{#1}}
\newcommand{\mb}[1]{\mathbb {#1}}

\newcommand{\ti}[1]{\widetilde {#1}}

\DeclareRobustCommand{\stirling}{\genfrac\{\}{0pt}{}}

\theoremstyle{plain}
\newtheorem{thm}{Theorem}[section]
\newtheorem*{thm*}{\bf Theorem }

\newtheorem{prop}[thm]{Proposition}
\newtheorem*{prop*}{\bf Proposition}

\theoremstyle{definition}
\newtheorem{defn}[thm]{Definition}
\newtheorem{conj}[thm]{Conjecture}
\newtheorem{example}[thm]{Example}

\theoremstyle{remark}
\newtheorem{rem}[thm]{Remark}

\numberwithin{equation}{section}

\title[Moduli spaces of colored graphs]{Moduli spaces of colored graphs}

\author{Marko Berghoff} \email{berghoff@math.hu-berlin.de}

\author{Max M{\"u}hlbauer} \email{muelbaum@physik.hu-berlin.de}

\begin{document}

\begin{abstract}
We introduce moduli spaces of colored graphs, defined as spaces of non-degenerate metrics on certain families of edge-colored graphs. Apart from fixing the rank and number of legs these families are determined by various conditions on the coloring of their graphs. The motivation for this is to study Feynman integrals in quantum field theory using the combinatorial structure of these moduli spaces. Here a family of graphs is specified by the allowed Feynman diagrams in a particular quantum field theory such as (massive) scalar fields or quantum electrodynamics. The resulting spaces are cell complexes with a rich and interesting combinatorial structure. We treat some examples in detail and discuss their topological properties, connectivity and homology groups.
\end{abstract}
\maketitle

\section{Introduction}

The purpose of this article is to define and study moduli spaces of colored graphs in the spirit of \cite{hv} and \cite{chkv} where such spaces for uncolored graphs were used to study the homology of automorphism groups of free groups. Our motivation stems from the connection between these constructions in geometric group theory and the study of Feynman integrals as pointed out in \cite{bk-cr}. Let us begin with a brief sketch of these two fields and what is known so far about their relation.

\subsection{Background} 
The analysis of Feynman integrals as complex functions of their external parameters (e.g.\ momenta and masses) is a special instance of a very general problem, understanding the analytic structure of functions defined by integrals. That is, given complex manifolds $X,T$ and a function $f:T \rightarrow \mb C$ defined by
\begin{equation*}
 f(t)=\int_{\Gamma} \omega_t,
\end{equation*}
what can be deduced about $f$ from studying the integration contour $\Gamma \subset X $ and the integrand $\omega_t \in \Omega(X)$?

There is a mathematical treatment for a class of well-behaved cases \cite{pham}, but unfortunately Feynman integrals are generally too complicated to answer this question thoroughly. They have however enough (combinatorial) structure allowing for the deduction of partial results by various methods, although many have not yet been put on a rigorous mathematical footing.\footnote{For what is known, \cite{analsm} is the classic reference for physicists while \cite{hwa} advocates a more mathematical point of view.} For example, Cutkosky's theorem \cite{cr}, relating the imaginary part of a Feynman integral to a simpler integral over generalized residues (in physics terms, certain \textit{edge propagators put on mass-shell}), was just recently proven in \cite{bk-cr}. Along the way, Bloch and Kreimer mention the idea of studying \textit{Outer space} and variants thereof to gain new insights into the analytic structure of Feynman integrals.

Outer space $CV_n$ is a topological space that arises in geometric group theory where it is used to study the automorphism group of a free group $F_n$. Inspired by ideas from Teichm\"uller theory, it is constructed as a space in which points are equivalence classes of marked metric graphs. The group $\mathrm{Out}(F_n)=\mathrm{Aut}(F_n)/\mathrm{Inn}(F_n)$ outer automorphisms of $F_n$ acts on $CV_n$ by permuting these markings \cite{cv}. The quotient of $CV_n$ by this action is a moduli space for connected metric graphs of rank $n$ without vertices of valence one or two (or genus $n$ \textit{tropical curves}, see \cite{caporaso}). Moreover, it is a rational classifying space for $\mathrm{Out}(F_n)$, hence computes the group homology $H_*(\mathrm{Out}(F_n);\mb Q)$.  

There are many homotopy equivalent models for Outer space. For instance, $CV_n$ deformation retracts to its so-called \textit{spine}, a simplicial complex whose elements naturally assemble into cubes. This produces a cubical complex with its elements represented by pairs $(G,F)$ of (marked) graphs  and spanning forests $F\subset G$. Furthermore, the group action behaves nicely on the spine of Outer space, allowing to set up a cubical chain complex to compute $H_*(\mathrm{Out}(F_n);\mb Q)$ as the (cubical) homology of the moduli space of rank $n$ graphs. 

The connection to physics, as proposed by Bloch and Kreimer, is established by the fact that this cubical complex also captures the combinatorial structure of \textit{cut} and \textit{reduced graphs} which show up in the study of Feynman integrals as complex functions of their external parameters. More precisely, to any pair $(G,F)$ where $F$ is a spanning forest of $G$ they associate two new graphs, a \textit{reduced graph} obtained by collapsing all edges of $G$ which do not connect different components of the spanning forest $F$ and a \textit{cut graph} where all those edges connecting different components are put \textit{on-shell} (see \cite{dklpz} for details and examples). This data allows to analyse a Feynman integral through the study of simpler integrals, i.e.\ to determine its singular loci and branch cuts together with their associated discontinuities.

On the other hand, the same data describes the boundary operator in the cubical chain complex for $CV_n$ or $CV_n / \mathrm{Out}(F_n)$, respectively.
Hence, the combinatorial topology of these spaces seems to be related to the analytic structure of Feynman integrals. So far this holds at least for every cell, i.e.\ every pair $(G,F)$ (see \cite{bk-cr,dklpz}), understanding the relations between neighboring cells is the main motivation for the present article. 

\subsection{Moduli spaces of colored graphs} In contrast to the elements in Outer space or the moduli space of graphs, physicists usually consider Feynman diagrams as graphs with additional structure. Depending on a chosen theory one needs to distinguish between different masses or particle types assigned to graph edges. Moreover, there are rules for which particles may interact, so that not all vertex types will be allowed. Such additional data can be represented, as a first approximation, by coloring the edges of a graph. Adjusting the definitions to this case we obtain moduli spaces parametrizing (isomorphism classes of) colored metric graphs. 
These spaces share structures similar to the uncolored case, e.g.\ the cubical decompositions described above, allowing to mimic the ideas of \cite{hv} to compute their homology groups algorithmically.

The idea (or hope) is that understanding the combinatorics and topology of these spaces will give new insight into the study of Feynman integrals. A closer look at this connection and its applications will be pursued in future work, including a comparison to the algebraic approach to Feynman integrals via \textit{coactions} that has received quite some attention lately (see for instance \cite{coact}). 

In the present paper we focus on the mathematical properties of these spaces which are pretty interesting in their own right. 
While we give rather general definitions, our concrete results are mainly focused on three special types of moduli spaces for rank one graphs
\begin{enumerate}
 \item with arbitrary colorings of their edges by a fixed number of colors $m$, 
 \item which are \textit{holocolored}, i.e.\ each edge has a different color,
 \item which are holocolored with the additional feature that edges keep the information of their coloring when collapsed. 
 \end{enumerate}
 In physics terms we think of the first type as a moduli space for all possible Feynman diagrams in a theory with $m$ different scalar particles. The second type models the ``generic case'' thereof where only diagrams with different masses occur (this case is much simpler from an analytic perspective). The third space of holocolored graphs with ``remembered edges'' is built to model a diagrammatic account of the \textit{operator product expansion} (see for example \cite{lh75} or \cite{iz}). Here, in contrast to the analysis of a Feynman integral by looking at its cut and reduced graphs, the type of a new vertex formed by collapsing an edge depends on the label of the latter. To model this effect we therefore consider a space of graphs where edges of zero length still carry a color.

\subsection{Results} In the three above mentioned cases we study the homology groups of these spaces. Our main computational tool, generalizing the ideas of \cite{hv} to the colored case, is a cubical chain complex that allows to compute their rational homology with computer assistance. The results can be found in Tables \ref{t:homdimc}, \ref{t:hom_dirk} and \ref{t:hom_dirk2} below.
Moreover, we discuss some special cases in detail and derive a couple of general results for the case of rank one graphs, showing that 
\begin{itemize}
\item all moduli spaces of arbitrary colored graphs are simply connected. 
\item the computer results suggest that the homology groups of these moduli spaces are independent of the number of allowed colors, except in the top dimensional rank. We comment on some partial results to prove this conjecture.  
\item the highest non-trivial Betti number of the moduli space of $m$-colored graphs with $s$ legs is given by a polynomial in $m$ of degree $s$. 
 \item for both types of moduli spaces of holocolored graphs with $s$ legs we obtain the highest non-trivial homology group by an explicit construction, showing that it is isomorphic to $\mathbb{Z}^{\frac{1}{2}(s-1)!}$.
 \item for all three types of moduli spaces we calculate the Euler characteristic. 
\end{itemize}

\subsection{Organisation} Section \ref{s:bg} serves as a quick reminder on graphs and cell complexes, setting up the necessary definitions and notational conventions used throughout this work. In Section \ref{s:ms} we recall the notion of moduli spaces of graphs in the classical sense and discuss their most important properties. Then we define moduli spaces of colored graphs formally and consider several examples of special colorings in detail. The remaining two sections are concerned with the topology of these spaces. In Section \ref{s:mc} we study the case of arbitrarily colored graphs. For this we introduce a cubical chain complex that computes the rational homology of these spaces. Then we list and discuss the results of these calculations for the rank one case and prove some general statements on homological and homotopical properties of these spaces.
Section \ref{s:hc} deals with holocolored graphs, with or without remembered edges. We display the results of computer calculations in the cubical complex as well as some direct results obtained by algebraic and combinatorial methods, i.e.\ calculation of the top dimensional homology groups and the Euler characteristics. 
\newline

\textbf{Acknowledgements.} Both authors owe many thanks to Dirk Kreimer, not only for sparking interest in this topic with his original ideas, but also for steady support and encouragement. We also thank Sam Yusin for helpful discussions. M.B.\ thanks Karen Vogtmann for all the valuable advice and discussions and her hospitality during his stay at the \textit{2017-2018 Warwick EPSRC Symposium on Geometry, Topology and Dynamics in Low Dimensions}.

\section{Preliminaries}\label{s:bg}
Let us start by introducing the basic definitions and notational conventions for the central objects in this article, graphs and cell complexes. This is standard material, but as some definitions vary slightly in the literatur, we recall them here.

\subsection{Graphs}
Graphs are very versatile mathematical objects that show up in a variety of fields including discrete mathematics, computer science and quantum field theory. They prominently arise in the perturbative approach to the latter in form of Feynman diagrams which represent integrals contributing to probability amplitudes in high-energy particle scattering processes (see for example \cite{iz}).

There are various possibilities to define graphs,\footnote{Here we always consider \textit{non-empty finite multi-graphs}.} each suited for different purposes. Physicists often use a definition based on half-edges which allows for a distinction between internal and external edges as needed for Feynman diagrams.  

\begin{defn}
 A \textit{graph} is a tuple $G=(V,H,s,c)$ consisting of a set of vertices $V=V(G)$, a set of half-edges $H=H(G)$, a map $s=s_G:H \rightarrow V$ which connects each half-edge to its source vertex and a map $c=c_G:H \rightarrow H$ with $c^2=\text{id}_H$ that connects half-edges with each other; if $h_1 \neq h_2\in H$ are two distinct half-edges with $c(h_1)=h_2$, the pair $\{h_1,h_2\}$ is called an (internal) edge\footnote{By using ordered pairs $(h_1,h_2)$ one obtains oriented edges; $G$ is then called a \textit{directed graph}.} of $G$, otherwise $h_1=h_2$ is called an external edge, leg or hair. We denote the set of internal edges of $G$ by $E(G)$.
  
A subgraph $\gamma \subset G$ is a graph such that $V(\gamma) \subset V(G)$, $H(\gamma)\subset H(G)$ and $s_\gamma=s_G|_{H(\gamma)}$, $c_\gamma=c_G|_{H_\gamma}$. Subgraphs can have external legs, but we will only consider so-called internal subgraphs, i.e.\ subgraphs without legs.
\end{defn}

The following notations occur frequently throughout this article.
\begin{itemize}
\item The \textit{rank} or \textit{loop number} of $G$ is denoted by $|G|$ or $h_1(G)$ (the first Betti number of $G$ viewed as a CW-complex).
\item For any vertex $v\in V(G)$ we denote its \textit{valency} by $|v|:=|s^{-1}(v)|$.
\item $G$ is called \textit{one-particle irreducible (1PI)}, \textit{bridge-free} or \textit{core} if it is connected and still connected upon removal of any internal edge $e\in E(G)$. In case $G$ is not 1PI, the edges leaving the graph disconnected upon removal are called \textit{bridges} or \textit{separating edges}.
\item A graph $G$ is called \textit{admissible} if it is 1PI and $|v|\geq3$ for all $v\in V(G)$.
\item A graph $F$ is called a \textit{$k$-forest} if it has no loops, i.e.\ $|F|=0$, and $k$ connected components. In particular, a 1-forest is called a \textit{tree}. A subgraph $F\subset G$ is called a \textit{spanning $k$-forest} if $F$ is a $k$-forest and $V(F)=V(G)$. If $k=1$, then $F$ is said to be a \textit{spanning tree}.
\item A graph with a single vertex, $n$ internal edges and $s$ legs is called a \textit{rose with $n$ petals (and $s$ thorns)} and denoted by $R_{n,s}$.
\end{itemize}

Sometimes, especially for topological considerations, it will be more convenient to think of graphs as one dimensional $CW$-complexes. In this case legs can be modeled as attached to univalent vertices or as additional labels (basepoints) on some vertices of $G$ (the definition of admissibility has then to be adjusted, allowing for univalent vertices or for labeled bivalent vertices, respectively). 
\newline

In the following we will need two operations on graphs, removing and contracting edges.
\begin{defn}
Let $G$ be a graph and $\gamma\subset G$ an (internal) subgraph.
\begin{itemize}
\item $G \backslash\gamma$ denotes the graph $G$ with all edges of $\gamma$ removed, i.e.\ $V(G\backslash\gamma)=V(G)$ and $E(G\backslash\gamma)=E(G)\setminus E(\gamma)$.
\item $G/ \gamma$ is the graph $G$ with all edges of $\gamma$ collapsed: For $\gamma$ connected, $G/ \gamma$ is obtained from $G$ by collapsing $\gamma$ to a single vertex, i.e.\ $\gamma$ is replaced by a vertex with all edges connecting $\gamma$ and $G\backslash \gamma$ attached to it. 
If $\gamma$ is not connected, $G/ \gamma$ is defined by shrinking each connected component in this manner.
\end{itemize}
\end{defn}
The case where the subgraph $\gamma$ is a forest will occur frequently in this text. In that case the number of loops does not change when contracting $\gamma$, $|G /\gamma|=|G|$. In particular, if $\gamma$ is a spanning tree, then $G/\gamma$ is a rose with $|G|$ petals.
\newline

So far we have considered graphs as purely combinatorial objects. To obtain more structure, a graph can be endowed with a metric by means of a length function that assigns to each edge a positive real number. 
\begin{defn}
A \textit{metric graph} is a pair $(G,\lambda)$ where $G$ is a graph and $\lambda$ a map $\lambda:E(G)\to\mathbb{R}_{>0}$ giving each edge of $G$ a length. The \textit{volume} of $G$ is defined as the sum of all edge lengths,
\begin{equation*}
 \mathrm{vol}_\lambda(G):=\sum_{e\in E(G)}\lambda(e).
\end{equation*}
\end{defn}

Thus, given a metric graph $G$ the distance between two points can be defined as the minimum length of a path connecting them, turning $G$ into a metric space.

If we additionally allow for an edge $e\in E(G)$ to have zero length, we typically identify $(G,\lambda)$ with the contracted graph $(G/e,\lambda_{|E(G)\setminus\{e\}})$.
\newline

Lastly, we need a notion of maps between graphs that makes sense in both the combinatorial and topological setting.
\begin{defn}
Let $G,G'$ be graphs. By a map $f: G \rightarrow G'$ we mean a cellular map between $G$ and $G'$ viewed as CW-complexes (see Definition \ref{defn:cw} below).
In addition, we require $f$ to map legs to legs or basepoints to basepoints, respectively.
\end{defn}

\subsection{Cell complexes}
All spaces we encounter in this work are cell complexes, built from cells of various types. In this section we give a short account of these types, taking a geometric approach following \cite{ah} and \cite{switzer}.

\subsubsection{Semi-simplicial and cubical complexes} First, we consider spaces that decompose into simplices. 

\begin{defn}\label{def:deltacomplex}
A topological space $X$ together with a collection of continuous maps $\sigma_\alpha:\Delta^n\to X$ (with $n\in\mathbb{N}$ dependent on $\alpha$) is called a \textit{semi-simplicial} or \textit{$\Delta$-complex} if
\begin{itemize}
\item All restrictions of $\sigma_\alpha$ to the interior of $\Delta^n$ are injective such that each $x\in X$ is in the image of exactly one such restriction.
\item For all $\sigma_\alpha:\Delta^n\to X$ the restriction to any face is a map $\sigma_\beta:\Delta^{n-1}\to X$.
\item For any $\alpha$, any $A\subset X$:\, $\sigma_\alpha^{-1}(A)$ is open $\iff$ $A$ is open.
\end{itemize}
\end{defn}

Obviously, one may take also other types of building blocks to decompose a given space. For instance, by replacing simplices with cubes. This method has certain computational advantages, some of which are utilized in Sections \ref{s:mc} and \ref{s:hc}. 

Let $I$ denote the interval $[0,1]\subset\mathbb{R}$ and define the \textit{standard $n$-cube} as the product of intervals $\square^n:=I^n=[0,1]\times\ldots\times[0,1]$ and $\square^0 :=\{0\} $. 
Given a topological space $X$ we want to decompose it into an union of such cubes. The direct analog to a simplicial complex is to restrictive for the purpose of this work. Therefore, we mimic the above definition of a $\Delta$-complex:

\begin{defn}
 A \textit{(singular) $n-$cube} in $X$ is a continuous map $\sigma:\square^n\rightarrow X$. If $\sigma$ is injective, the cube is called \textit{regular}, otherwise \textit{degenerate}. 

 The \textit{$i$-th (primary) faces} of a cube $\sigma: \square^n \rightarrow X$ are defined as the maps
 \begin{equation*}
    f^i_+\sigma:\square^{n-1}\longrightarrow X, \ (x_1, \ldots, x_n) \longmapsto \sigma(x_1, \ldots, x_{i-1},0,x_{i+1},\ldots, x_n) 
 \end{equation*}
and
\begin{equation*}
 f^i_-\sigma :\square^{n-1}\longrightarrow X, \ (x_1, \ldots, x_n) \longmapsto \sigma(x_1, \ldots, x_{i-1},1,x_{i+1},\ldots, x_n).
\end{equation*}
\end{defn}

\begin{defn}
A topological space $X$ together with a collection of continuous maps $\sigma_\alpha:\square^n\rightarrow X$ (with $n\in\mathbb{N}$ dependent on $\alpha$) is called a \textit{cubical complex} if
\begin{itemize}
\item All restrictions of $\sigma_\alpha$ to the interior of $\square^n$ are injective such that each $x\in X$ is in the image of exactly one such restriction.
\item For every $\sigma_\alpha:\square^n\rightarrow X$ the restriction to any primary face is a map $\sigma_\beta:\square^{n-1}\rightarrow X$.
\item For any $\alpha$ and any $A\subset X$:\, $\sigma_\alpha^{-1}(A)$ is open $\iff$ $A$ is open.
\end{itemize}
\end{defn}

For a more detailed treatment of cubical complexes the interested reader is referred to \cite{massey}.

\subsubsection{CW-complexes} By replacing simplices or cubes by disks $D^n:=\{ x\in \mb R^n \mid ||x||\leq 1 \}$ and allowing much more general gluing maps we obtain the notion of \textit{CW-complexes}. Instead of a formal definition, we give a building recipe for these kind of spaces (as in \cite{ah}, for a precise definition see \cite{switzer}).

\begin{defn}\label{defn:cw}
 A \textit{CW- or (regular) cell complex} is a space $K$ constructed inductively as follows.
 \begin{itemize}
  \item Start with a discrete set $K^{(0)}$, the set of \textit{$0$-cells} of $K$.
  \item Inductively form $K^{(n)}$ from $K^{(n-1)}$ by attaching $n$-cells $D^n_\alpha$ via maps continuous maps $\varphi_\alpha:S^{n-1} \approx \partial D^n_\alpha \rightarrow K^{(n-1)}$. Thus, $K^{(n)}$ is the quotient space of $K^{(n-1)} \sqcup_\alpha D^n_\alpha$ under the relation $x \sim \varphi_\alpha(x)$ for $x \in \partial D^n_\alpha$. The \textit{$n$-cells} $e^n_\alpha$ of $K$ are the homeomorphic images of $ D^n_\alpha \setminus \partial D^n_\alpha$ under this quotient map.
  \item $K= \cup K^{(n)}$ with the weak topology: A set $U\subset K$ is open if and only if $U\cap K^{(n)}$ is open for all $n$.  
 \end{itemize} 
The spaces $K^{(n)}$ are referred to as the \textit{$n$-skeleta} of $K$. Given two CW-complexes, a continuous map $f:K \rightarrow L$ is called \textit{cellular} if $f(K^{(n)}) \subset L^{(n)}$ for all $n$.
\end{defn}

Clearly, cell complexes provide a very flexible setting to deal with topological spaces. One major advantage is that for a CW-decomposition of a given space much fewer cells are needed than in the simplicial or $\Delta$ setting. 

\subsection{Cubical homology} Given any of the above types of decompositions of a space $X$ there is a corresponding chain complex whose homology is isomorphic to the singular homology of $X$. Simplicial and cellular homology are well-known, but one may also use cubes to calculate $H_*(X)$. The chain complex $(C_*^\square(X),\partial_*)$ associated to a cubical complex $(X,\{\sigma_\alpha\}_{\alpha\in A})$ is constructed by defining the chain groups $C_n^\square(X)$ to be the free abelian groups generated by all (regular) cubes $\sigma_\alpha:\square^n\rightarrow X$. The boundary morphism $\partial_n^\square$ acts linearly on the $n$-chains and its action on a single generator $\sigma_\alpha$ is defined by
\begin{equation}\label{eq:cubedelta}
\partial^\square \sigma_\alpha := \partial_+^\square \sigma_\alpha + \partial_-^\square \sigma_\alpha,
\end{equation}
where
\begin{align*}
\partial_+^\square \sigma_\alpha  &:= \sum_{i=1}^{n}(-1)^{i-1} f_+^i \sigma_\alpha =\sum_{i=1}^{n}(-1)^{i-1} \sigma_\alpha|_{I^{i-1} \times \{0\} \times I^{n-i}},\\
\partial_-^\square \sigma_\alpha  & := \sum_{i=1}^{n}(-1)^{i} f_-^i \sigma_\alpha =\sum_{i=1}^{n}(-1)^i \sigma_\alpha|_{I^{i-1} \times \{1\} \times I^{n-i}}.
\end{align*}

This produces a chain complex whose homology is isomorphic to the singular homology of $X$.
It must be remarked though that caution is required when dealing with degenerate cubes, i.e.\ cubes $\sigma:\square^n\rightarrow X$ which are not injective (cf.\ \cite{massey}). The treatment of this technicality is omitted here since it does not occur in the cases considered in this work.

\section{Moduli spaces of graphs}\label{s:ms}
Moduli spaces of (uncolored) graphs can be defined as quotients of \textit{Culler-Vogtmann Outer space} $CV_n$ and generalizations thereof. Points in the latter are tuples $(G,\lambda,g)$ where $(G,\lambda)$ is a metric graph of rank $n$ and $g$ a \textit{marking}, a (homotopy class of a) homotopy equivalence between $G$ and the rose graph $R_n$ \cite{cv}. Roughly speaking, Outer space $CV_n$ is a moduli space of marked metric graphs that is equipped with an action of $\mathrm{Out}(F_n)$ which acts by changing the markings. Defining a moduli space of graphs as the orbit space of this action has certain advantages,\footnote{The same holds for an algebro-geometric approach via tropical curves, see for example the survey in \cite{caporaso}.} but for the sake of brevity and having the application to Feynman diagrams in mind we stick to a more direct definition. Nevertheless, our construction is heavily inspired and conceptually quite close to the case of Outer space and its generalizations.

For more on the ``approach from Outer space'' we refer to the survey in \cite{v-topgeo}.

\subsection{The uncolored case}

For $n>1$ let $(G,\lambda)$ denote an admissible graph $G=(V,E)$ of rank $n$ and without legs, employed with a metric $\lambda: E \rightarrow \mb R_{\geq 0}$. We define an equivalence relation on the set of such metric graphs by declaring $(G,\lambda)$ to be equivalent to $(G',\lambda')$ if the metrics differ only by a scaling factor. More precisely, if $Z\subset E$ and $Z' \subset E'$ denote the sets of edges in $G$ and $G'$ on which $\lambda$ and $\lambda'$ vanish, then $\lambda' \circ \varphi \propto  \lambda$ for an isomorphism $\varphi : G/Z \to G'/Z'$,
\begin{equation}\label{eq:equiv}
 (G,\lambda) \sim (G',\lambda') \Longleftrightarrow \exists \ \varphi: G/Z \xrightarrow{\sim} G'/Z' , c>0  \text{ s.t. } \lambda' \circ \varphi= c \text{ on } E\setminus Z.
\end{equation}

This relation allows to consider metric graphs with their volume normalized to one. In addition, if $\lambda$ vanishes on some $Z\subset E$ we identify $(G,\lambda)$ with the contracted graph $G/Z$ and $\lambda_{|E\setminus Z}$ normalized appropriately. If $\lambda$ vanishes only on forests $F\subset E$, the metric is called \textit{regular}, otherwise \textit{degenerate}.

\begin{defn}
 The \textit{moduli space of rank $n$ graphs} $\m {MG}_{n}$ is defined as
 \begin{equation*}
  \m {MG}_n := \big\{ (G,\lambda) \mid G \text{ admissible with } |G|=n, \lambda \text{ a regular metric on }G \big\}_{ \big/\sim}.
 \end{equation*}
\end{defn}

As a topological space $\m {MG}_n$ is best understood by first considering the space 
\begin{equation*}
  K_n = \bigsqcup_{G \in \m G_n} \sigma_G 
\end{equation*}
where $\m G_n=\{ [G] \mid G \text{ admissible and }|G|=n \}$ is the set of isomorphism classes\footnote{In the following we omit the notation $[\cdot]$, tacitly picking representatives of each class. Note that some constructions demand for exchanging these representatives, e.g.\ to make sense of $G=G'/F$.} of admissible rank $n$ graphs and $\sigma_G$ an open $(|E|-1)$-dimensional simplex, the interior of $\Delta_G=\{ (x_e)_{e\in E} \mid \sum x_e=1\}$.

A face of $\Delta_{G}$ lies in $K_n$ if and only if the edge set $\{e \in E \mid x_e=0\}$ forms a forest in $G$. On the other hand, some faces of $\Delta_G$ may be missing, namely those corresponding to edge variables $x_e$ vanishing on subgraphs $\gamma \subset G$ with $|\gamma|>0$. Points in these faces are said to \textit{lie at infinity}.

Thus, the open simplices of $K_n$ are glued together using the face relations
\begin{equation*}
 \sigma_G \subset \overline{\sigma_{G'}} \Longleftrightarrow \exists F\subset G' \text{ a forest with } G'/F=G,
\end{equation*}
so that $K_n$ is a \textit{relative simplicial complex}, i.e.\ a pair $K=(X,Y)$ where $X$ is a simplicial complex and $Y\subset X$ a subcomplex, such that $K=X\setminus Y$ (cf.\ \cite{stanley87}).

Each $G$ represents also a cell in $\m {MG}_n$. It is given by taking the quotient with respect to the relation which identifies a point in $\sigma_G$ with a regular metric $\lambda$ on $G$, normalized to unit volume. If all edges of $G$ were distinguishable, this relation would be one-to-one, but if $G$ has nontrivial automorphisms, then this operation ``folds the simplex onto itself'', see for example Figure \ref{fig:simpfold}.

\begin{figure}
\centering
\includegraphics[width=8cm]{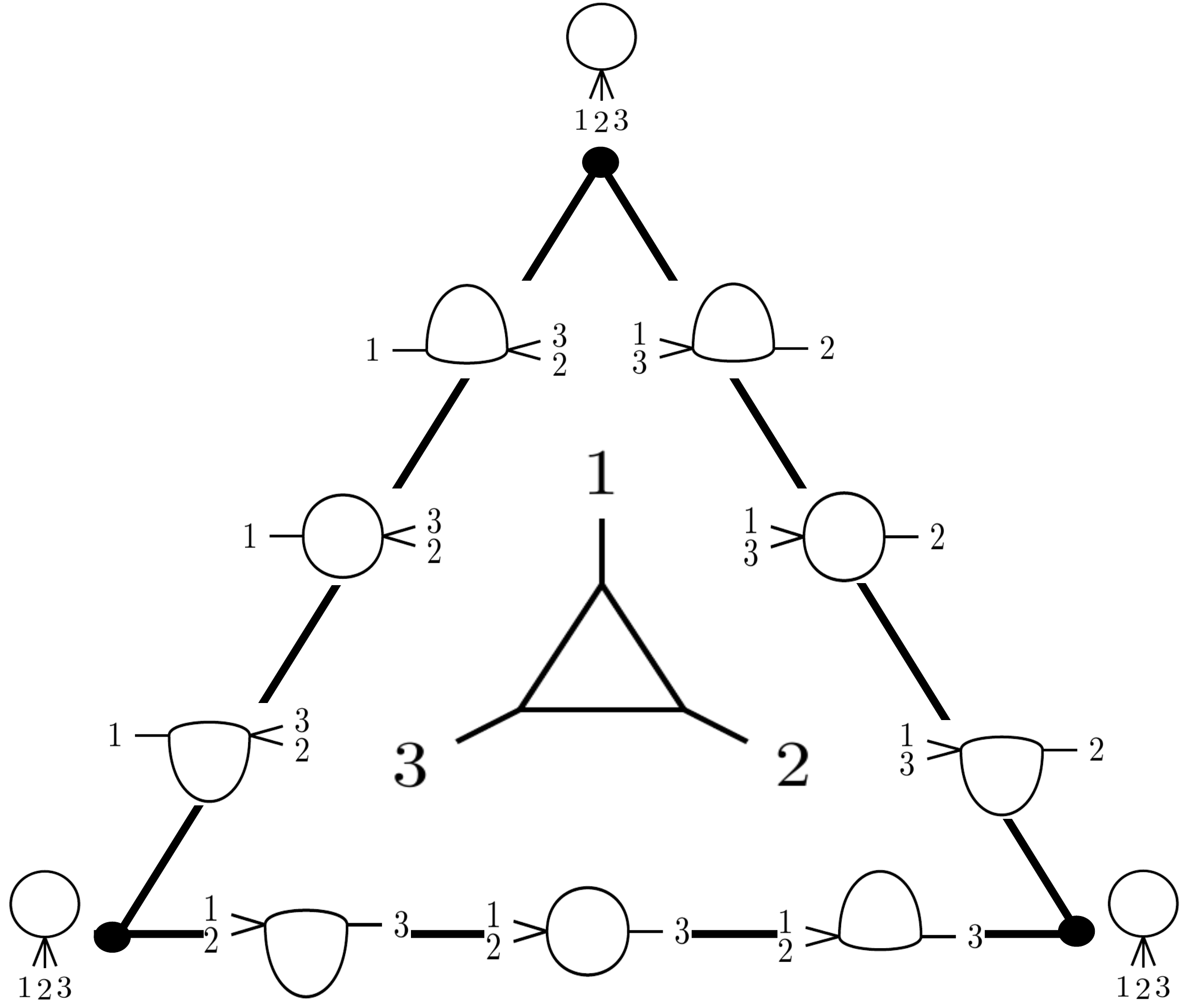}
\caption{$\m {MG}_{1,3}$: Here all three vertices and opposite points in each of the three edges are identified. Hence, $\m {MG}_{1,3}\cong S^2$.}
\label{fig:simpfold}
\end{figure}

This is however the only bad thing that can happen, so $\m {MG}_n$ can be described as a CW-complex with missing cells. Since an admissible graph of rank $n$ can have at most $3n-3$ edges (all vertices trivalent), we find $\dim K_n=\dim \m {MG}_n=3n-4$. 
\newline

Another way to understand the topology of $\m {MG}_n$ is to first replace $K_n$ by a subspace $SK_n$ to which it deformation retracts and then take the quotient as described above. The space $SK_n\subset K_n$ is a simplicial complex and plays the same role for $K_n$ as does the \textit{spine} for Outer space. It can be defined as the geometric realization of the poset $(\m G_n,\preceq)$ where
\begin{equation}
 \m G_n :=\{ G \mid G \text{ adm.}\wedge|G|=n \},\   G \preceq G' \Longleftrightarrow \exists F \subset G' \text{ a forest} : G=G'/F. 
\end{equation}
Hence, a $k$-simplex in $SK_n$ is represented by a chain $G_0 \preceq \ldots \preceq G_k$ of $k+1$ graphs in the poset $\m G_n$.

Since an admissible graph can have at most $3n-3$ edges, we find $\dim SK_n=2n-3$ as the maximal number of edges of its spanning trees.
\newline

The whole construction and the decompositions described above naturally generalize to the case of graphs with $s$ external edges. 
Here we think of these legs as labeled basepoints (at the vertices of $G$) and adjust the definition of admissibility accordingly. One then considers basepointed metric graphs $(G,\{v_1,\ldots, v_s\},\lambda)$ and declares two such tuples to be equivalent if and only if there is a basepoint-preserving isomorphism $\varphi: (G/Z,\{v_1,\ldots, v_s\}) \xrightarrow{\sim} (G'/Z',\{w_1,\ldots, w_s\})$ and $c>0$ such that $\lambda' \circ \varphi = c \lambda$ on $E\setminus Z$ (as above, $Z\subset E$ and $Z' \subset E'$ denote the sets of edges on which $\lambda$ and $\lambda'$ vanish).

\begin{defn}\label{def:mns}
For $n,s \in \mb N$ the \textit{moduli space of rank $n$ graphs with $s$ legs} $\m {MG}_{n,s}$ is defined as 
\begin{equation*}
 \m {MG}_{n,s} := \left \lbrace (G,\{v_1,\ldots, v_s\},\lambda) \; \middle| \;
  \begin{tabular}{@{}l@{}} G admissible with $s$ legs $v_1,\ldots,v_s$, \\$n$ loops and $\lambda:E\rightarrow \mb R_{\geq 0}$ regular
  \end{tabular}
  \right \rbrace_{\big/\sim}.
 \end{equation*}
\end{defn}

For $s>0$ these moduli spaces play the same role for a sequence of groups $\Gamma_{n,s}$ as does $\m {MG}_n$ for $\mathrm{Out}(F_n)$. Here $\Gamma_{n,s}$ is the group of (relative) homotopy classes of self-homotopy equivalences of a rank $n$ graph fixing its $s$ legs, say of the rose graph $R_{n,s}$ with $n$ petals and $s$ thorns,
\begin{equation*}
\Gamma_{n,s}:=\pi_0(\mathrm{aut}(R_{n,s}))  \Longrightarrow \Gamma_{n,0}=\mathrm{Out}(F_n), \Gamma_{n,1}=\mathrm{Aut}(F_n), \ldots
\end{equation*}
For a precise definition and further reading on these groups with applications in geometric group theory we refer to \cite{chkv} and the survey in \cite{v-topgeo}.

Similar to the case of $\m {MG}_n$ one might first look at the space
\begin{equation*}
 K_{n,s}:=\bigsqcup_{G \in \m G_{n,s}} \sigma_G 
\end{equation*}
where now $\m G_{n,s}$ denotes the set of isomorphism classes of admissible rank $n$ graphs with $s$ legs. It decomposes into an union of simplices with missing faces. For $s\geq 2$ it makes sense to consider the case $n=1$ as well for which the intersection of two simplices may be the union of more than one of their faces, cf.\ Figure \ref{fig:simpfold}. Therefore, $K_{n,s}$ forms in general a ``relative $\Delta-$complex'', a $\Delta$-complex with some of its simplices deleted. $\m {MG}_{n,s}$ is then obtained from $K_{n,s}$ by identifying points in its cells $\sigma_{G}$ with metrics $\lambda$ on the corresponding graphs $G$.

From the definition of admissibility and by an Euler characteristic argument we deduce $\dim K_{n,s}=\dim \m {MG}_{n,s} = 3n+s-4$. 

For $n>1$ there is a deformation retract $SK_{n,s} \subset K_{n,s} $, defined analogously to the case without legs, which is a simplicial complex of dimension $2n-3+s$. It may be used to set up a cubical chain complex that allows to compute the homology of $\m {MG}_{n,s}$ as was done in \cite{hv} for the case $s=1$. Moreover, these moduli spaces are rational classifying spaces for the groups $\Gamma_{n,s}$ so that these complexes compute the group homology $H_*( \Gamma_{n,s}  ; \mb Q)$.

\subsection{Moduli spaces of colored graphs}
In the previous section we used edge-metrics to define topological spaces populated by (admissible) graphs. When physicists draw Feynman diagrams to represent particle scattering processes, edges may describe different kinds of particles. To encode this additional piece of data we now consider colored edges.
\begin{defn}
Let $G=(V,E)$ be a graph and $m\in\mathbb{N}$. An \textit{$m$-coloring} of $G$ is a map $c:E \rightarrow \{1,2,...,m\}$.
\end{defn}

An $m$-coloring of a graph represents some (physical) property of the edges which can take $m$ different values. For example one might think of a Feynman diagram in a scalar field theory of three different massive particles as a graph endowed with a 3-coloring, one color for each particle in the theory.\footnote{In many actual quantum field theoretical calculations, the spin of particles adds an additional feature; half-integer spin particles come as oriented edges.}
\newline

Analogous to the constructions in the last section metric graphs endowed with a coloring can be represented by points in a moduli space of colored graphs. Here we will consider three different cases in detail: Spaces in which any $m$-coloring is allowed and spaces of \textit{holocolored} graphs in which only injective colorings with a fixed set of colors are admitted, so that each edge is assigned a different color. In the latter case two types of spaces are distinguished, one using colored graphs as before while the other deals with graphs which retain the information of their coloring upon shrinking edges. 

To make things precise, we start with a definition of equivalence for colored metric graphs. For this we define two relations by

\begin{equation*}
(G,\lambda,c) \sim (G',\lambda',c') \Longleftrightarrow 
\begin{cases}
\text{for } Z=\{\lambda=0\}\subset E,Z'=\{\lambda'=0\}\subset E' \\ \text{exists } 
 \varphi:G/Z \xrightarrow{\sim} G'/Z' \text{ with } \\
  c'= \varphi \circ c_{|E\setminus Z} \text{ and } \lambda' \circ \varphi =  \lambda_{|E\setminus Z}.
\end{cases}
\end{equation*}
and 
\begin{equation*}
(G,\lambda,c) \sim_* (G',\lambda',c') \Longleftrightarrow \begin{cases}
\text{ exists } \varphi:G \xrightarrow{\sim} G' \text{ with } c'= \varphi \circ c \text{ and } \\ 
\text{ for } Z=\{\lambda=0\},Z'=\{\lambda'=0\} \text{ exists } \\
\psi: G/Z \xrightarrow{\sim} G'/Z' \text{ with } \lambda'_{|E'\setminus Z'} \circ \varphi =  \lambda_{|E\setminus Z}.
\end{cases} 
\end{equation*}

Note that the defintion of $\sim$ allows to forget the color of an edge that is collapsed to zero length. Therefore, this is the appropriate generalization of the relation \eqref{eq:equiv} to the colored case and we keep using the same symbol. On the other hand, the relation $\sim_*$ requires matching colors also for edges of zero length. In addition, in the presence of legs we require all maps to preserve basepoints as in Definition \ref{def:mns}.

\begin{defn}
Let $n,s \in \mb N$ and define the \textit{moduli space of $m$-colored graphs} as 
\begin{equation*}
  \m {MCG}^m_{n,s} := \left \lbrace (G,\{v_1,\ldots, v_s\},\lambda,c) \; \middle| \;
  \begin{tabular}{@{}l@{}} G adm.\ with $s$ legs, $|G|=n$, \\ $\lambda$  regular, $c: E \rightarrow \{1, \ldots, m\}$
  \end{tabular}
  \right \rbrace_{\Big/\sim}.
\end{equation*}
\end{defn}

\begin{defn} 
For $n,s  \in \mb N$ set $C:=\{ 1, \ldots, 3(n-1)+s \}$ and define
\begin{enumerate}
\item[-] the \textit{moduli space of holocolored graphs} by
\begin{equation*}
  \m {MHG}_{n,s} := \left \lbrace (G,\{v_1,\ldots, v_s\},\lambda,c) \; \middle| \;
  \begin{tabular}{@{}l@{}} G adm.\ with $s$ legs, $|G|=n$, \\ $\lambda$  regular, $c: E \rightarrow C$ injective
  \end{tabular}
  \right \rbrace_{\Big/\sim}.
\end{equation*}
\item[-] the \textit{moduli space of holocolored graphs with remembered edges} by
\begin{equation*}
  \m {MRG}_{n,s} := \left \lbrace (G,\{v_1,\ldots, v_s\},\lambda,c) \; \middle| \;
  \begin{tabular}{@{}l@{}} G adm.\ with $s$ legs, $|G|=n$, \\ $\lambda$  regular, $c: E \rightarrow C$ injective
  \end{tabular}
  \right \rbrace_{\Big/\sim_*}.
\end{equation*}
\end{enumerate}
\end{defn}

If $n>1$, then arguing as in the uncolored case we see that each moduli space is the quotient of a space that decomposes into a disjoint union of open simplices, one for each isomorphism class of admissible colored graphs with $n$ loops and $s$ legs. In the case of $m$- and holocolorings the face relations are again given by contracting forests, but now with the additional requirement of matching colors. In the case of remembered edges this is a bit more delicate; here it is best to think of a point as given by either
\begin{itemize}
 \item an admissible colored metric graph with additional labels on its vertices that keep track of the contracted edges (cf.\ Section \ref{ss:re}), or
 \item an admissible colored combinatorial graph with all legs distinct and without identifying edges of length zero and contracted edges, $\lambda(e)=0 \not \Rightarrow G=G/e$,
\end{itemize}
and adjust the partial order accordingly.

In any case, for $n>1$ all moduli spaces contain deformation retracts that have the structure of simplicial or cubical complexes (the case $n=1$ is slightly different and will be considered in detail in Section \ref{ss:ccc}).

The only (and rather nontrivial) difference is that the notion of isomorphic graphs now depends on the coloring. In particular, the folding due to automorphisms does not occur if all edges are colored differently. Therefore, both spaces of holocolored graphs, $\m {MHG}_{n,s}$ and $\m {MRG}_{n,s}$, decompose directly into $\Delta$-complexes with missing faces. Their building blocks are the open simplices $\sigma_{(G,c)}$ as described above, one for each isomorphism class of admissible colored graphs $(G,c)$.
On the other hand, the moduli spaces of arbitrarily colored graphs $\m {MCG}_{n,s}^m$ have no such restrictions and contain thus rather ``wildly folded'' simplices. 
\newline

In the following sections we will study the topology of these three spaces in the one-loop case. This case has two advantages: No missing faces occur and the class of admissible graphs is very simple, allowing for direct calculations that would get quickly out of hand for graphs of higher rank. In addition, the analytic properties of Feynman integrals for one-loop diagrams are well understood whereas all other cases still remain quite mysterious. 

Before we continue  let us briefly describe two other moduli spaces that should also be of interest for the study of Feynman diagrams; a detailed analysis is reserved for future work.
\newline

\textbf{Directed graphs:}
A \textit{directed graph} is a graph $G$ together with an orientation of its edges, i.e.\ each edge $e\in E$ has a \textit{source} and \textit{target} vertex, denoted by $v_-(e)$ and $v_+(e)$, respectively. We encode this data by a map $o=(v_-,v_+): E \rightarrow V\times V$.

To define the notion of equivalence for directed metric graphs we henceforth assume that for a point $(G,\lambda,o)$ the metric $\lambda$ is strictly positive - otherwise we replace $G$ by $G/Z$ where $Z \subset E$ is the set of edges of zero length with respect to $\lambda$.

\begin{defn}
Let $\sim'$ denote the generalization of the equivalence relation \eqref{eq:equiv} to the case of (basepointed) directed metric graphs,
\begin{equation*}
(G,\lambda,o) \sim' (G',\lambda',o') \Longleftrightarrow 
\begin{cases}
(G,\lambda) \sim (G',\lambda') \text{ and for all } e\in E: \\
o(e)=(x,y) \Longrightarrow o'(\varphi(e))=(\varphi(x),\varphi(y)).
\end{cases}
\end{equation*}
Then for $n,s \in \mb N$ the \textit{moduli space of directed graphs} is defined as 
\begin{equation*}
  \m {MDG}_{n,s} := \left \lbrace (G,\{v_1,\ldots, v_s\},\lambda,o) \; \middle| \;
  \begin{tabular}{@{}l@{}} G adm.\ with $s$ legs, $|G|=n$, directed \\ by $o: E(G) \rightarrow V(G)^2$ and $\lambda$ regular 
  \end{tabular}
  \right \rbrace_{\Big/\sim'}.
\end{equation*}
\end{defn}

The definition of $\sim'$ is compatible with the equivalence of colored graphs, so that the above construction naturally generalizes to \textit{moduli spaces of colored directed graphs}. Moreover, if necessary, we may also allow only for some edges to be oriented and/or colored, although the notation blows up considerably in this case. 
\newline

\textbf{Restriction on vertex types:}
In a realistic quantum field theory not all types of particle interactions are allowed which translates in the language of Feynman diagrams to restrictions on the possible vertex types. 

For instance, in quantum electrodynamics there is only one interaction vertex, a trivalent node connecting the three edge types present in this theory, an electron, a positron and a photon. 
Hence, a corresponding moduli space, say of 3-colored graphs (ignoring orientations; in fact electrons and positrons call for directed edges as well), should have fewer top-dimensional cells. Namely those only having vertices connecting three different edge types. In lower dimensions other vertices should be admitted though since all graphs obtained by contracting edges in allowed graphs contribute to the analysis of the corresponding Feynman integrals (or to the operator product expansion).
In this exemplary case we should define a \textit{moduli space of QED Feynman diagrams} as $\m {MCG}_{n,s}^3 \setminus \m Z$ where $\m Z$ is the union of all top-dimensional open cells that correspond to graphs having forbidden vertices and all lower dimensional cells corresponding to graphs that cannot be obtained by shrinking edges in allowed graphs.

Again, using colored and/or directed edges is a mere technicality and can in principle be implemented as above.  

\begin{rem}
All but the last type of moduli spaces come with a canonical projection map to $\m{MG}_{n,s}$, defined by forgetting the additional structure put on graphs and inducing a surjection on homology. We discuss this map in the case of $m$- and holocolored graphs below.
\end{rem}

\section{$m$-colored graphs}\label{s:mc}
The primary focus of this and the following sections will be the calculation of the rational homology $H_*(X ; \mb Q)$ for $X=\m {MCG}_{1,s}^m$, $\m {MHG}_{1,s}$ and $\m {MRG}_{1,s}$. 
The main computational tool for this endeavor is a cubical chain complex that was used in \cite{hv} to calculate the rational homology of $\m {MG}_{n,1}$. For $n>1$ the method generalizes directly to the case of $s>1$ and colored graphs. For $n=1$ the same ideas work although coming from a slightly different setup.

\subsection{The cubical complex}\label{ss:ccc}

In the following let $X_s^m := \m {MCG}_{1,s}^m$ denote the moduli spaces of $m$-colored one-loop graphs with $s$ legs. We now describe a decomposition of $X_s^m$ into cubes. 

For any colored graph $(G,c)$ the set of regular metrics, normalized to volume one, describes a closed (no faces at infinity) simplex of dimension $d=|E|-1$. However, as discussed in the previous section, the images of these simplices under the quotient operation with respect to $\sim$ do not assemble themselves into a simplicial or semi-simplicial complex. The structure can be saved though by performing a barycentric subdivision before taking the quotient.

For this we start with the space
\begin{equation*}
K_{1,s}^m:= \bigsqcup_{(G,c) \in \m G_{1,s}^m } \Delta_{(G,c)},
\end{equation*}
$\m G_{1,s}^m$ denoting the set of isomorphism classes of admissible $m$-colored rank one graphs, and consider its barycentric subdivision $BK_{1,s}^m$. 
 \begin{figure}[h]
\centering
\includegraphics[width=8cm]{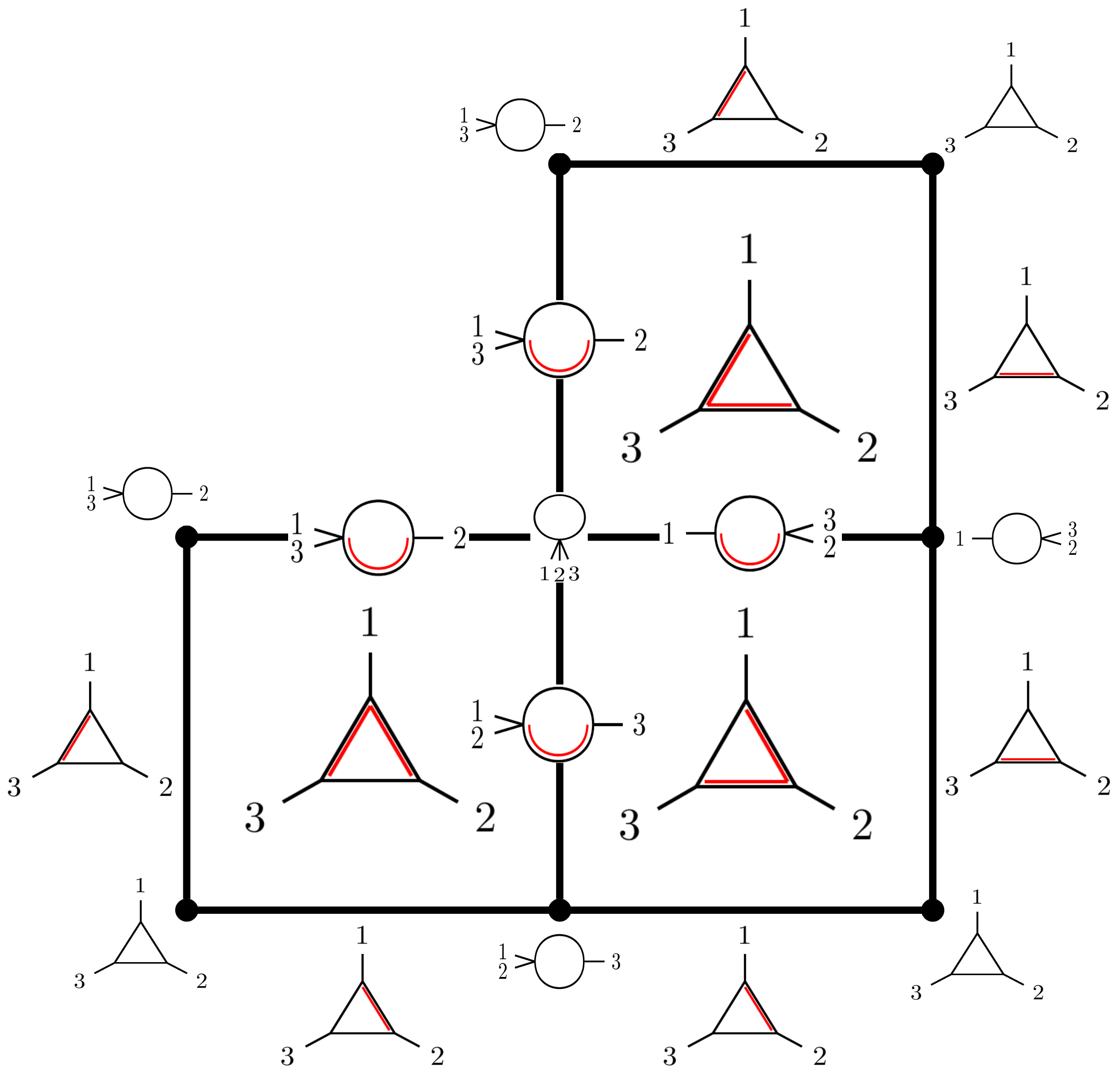}
\caption{A geometric representation of $X_3^1$ as a cubical complex.}
\label{fig:s3m1}
\end{figure}
A $k$-simplex in $BK_{1,s}^m$ is represented by a chain $(G_0,c_0) \preceq \ldots \preceq (G_k,c_k)$ of $k + 1$ colored graphs where the partial order $\preceq$ on $\m G_{1,s}^m $ is given by
\begin{equation*}
 (G,c) \preceq (G',c') \Longleftrightarrow \exists F\subset G' \text{ a forest with } G=G'/F \text{ and } c=c'|_{E'\setminus F}.
\end{equation*}

These simplices can be reassembled into cubes. In this picture a $k$-cube is given by a colored graph $(G,c)$ and a forest $F \subset G$ with $k=|E(F)|$ edges. This cube, denoted by $(G,F,c)$, is the collection of $k!$ simplices in $BK_{1,s}^m$ where each simplex is given by choosing an order $(e_1,\ldots, e_k)$ on the edges of $F$ and setting
\begin{equation*}
 G_0=G/F,\  G_1=G/(e_1, \ldots, e_{k-1}),\ \ldots, \ G_{k-1}=G/e_1, \ G_k=G
\end{equation*}
and
\begin{equation*}
 c_0=c|_{E\setminus E(F)},\  c_1=c|_{E\setminus \{ e_1, \ldots, e_{k-1} \}},\ \ldots, \ c_{k-1}=c|_{E\setminus \{e_1\}}, \ c_k=c.
\end{equation*}

This decomposition of $BK_{1,s}^m$ is fine enough to survive the quotient operation with respect to $\sim$, i.e.\ when identifying points in $\Delta_{(G,c)}$ with metrics $\lambda$ on $G$: No pair of points in a cube belong to the same equivalence class, so no folding occurs and only entire cubes can get identified with each other. 
Therefore, we find a similar decomposition of $X_s^m$ into cubes, one for each pair $(G,F,c)$ where $(G,c) \in \m G_{1,s}^m $ and $F\subset G$ a forest.

\begin{example}
Figure \ref{fig:s3m1} depicts the cubical decomposition of $X_3^1\cong \m{MG}_{1,3}$ (cf.\ Figure \ref{fig:simpfold}). Here the label of a cube $(G,F)$ is drawn as a graph with the edges of $F$ colored in red. Vertices and edges with the same label have to be identified.
\end{example}

\subsection{Rational homology}
The case without colors, or equivalently a single color $m=1$, serves as a good starting point to understand the computation and compare results with already established knowledge. 
A lot is already known about the homology of $\m {MG}_{n,s}$. In fact, the cases $s=0$ and $s=1$ encode the group homology of $\mathrm{Out}(F_n)$ and $\mathrm{Aut}(F_n)$, an active area of research. 
For the case $s>1$ the rational homology of $\m {MG}_{n,s}$ is fully determined for $n=1,2$ in \cite{chkv}. For $n=1$ we have
\begin{equation}\label{eq:homouter}
H_k(\m {MG}_{1,s};\mathbb{Q})=
\begin{cases}
    	\mathbb{Q}^{\binom{s-1}{k}} & \text{if $k\leq s-1$ is even}\\
    	0 & \text{otherwise.}
  	\end{cases}
\end{equation}

Interest in these homology groups stems from the fact that gluing together graphs of low rank along their legs induces so-called \textit{assembly maps} in homology which produce potential new homology classes in $\mathrm{Out}(F_n)$ and $\mathrm{Aut}(F_n)$ for larger $n$.

For the moduli space of $m$-colored graphs a cubical chain complex can be used to compute $H_*(\m {MCG}_{n,s}^m;\mb Q)$, for $n>1$ by generalizing the construction of \cite{hv} to the colored case and for $n=1$ by using the cubical complex described in the previous section.

In both cases we obtain similarly defined cubical chain complexes that compute the rational homology of $\m {MCG}_{n,s}^m$. The chain groups $C_k(\m {MCG}^m_{n,s})$ are the free abelian groups generated by all cubes $(G,F,c)$ where $|E(F)|=k$.\footnote{In fact, for $n>1$ not all cubes contribute. In this case the cubical chains are defined analogously to the cellular chains of a CW-complex; the ``cells'' are the quotients of cubes with respect to $\sim$ and under this operation some cubes become trivial in homology relative to their boundary. This does not happen in the case $n=1$, so we omit a discussion of this technicality.}
The boundary operator $\partial^\square$ from \eqref{eq:cubedelta} can easily be generalized to act on these triples $(G,F,c)$. We define the action of $\partial^m:C_*(\m {MCG}_{n,s}^m)\to C_*(\m {MCG}_{n,s}^m)$ on a cube $(G,F,c) \in C_k(\m {MCG}_{n,s}^m)$ by
\begin{equation}\label{eq:boundarycol}
\partial_k^m(G,F,c):=\sum_{i=1}^{k}(-1)^{i-1}\Big( (G,F\backslash e_i,c)-(G/e_i,F/e_i,c_{e_i}) \Big),
\end{equation}
where $c_{e_i}:=c|_{E \backslash \{e_i\}}$ is the coloring of $G$ with the edge $e_i$ collapsed. By construction this operator squares to zero.
\newline

From now on we stick to the one-loop case, i.e\ the spaces $X_s^m=\m MCG_{1,s}^m$. 
\begin{table}[]
\centering
\begin{tabular}{ c || c | c | c | c | c  }
   & $H_0$ & $H_1$ & $H_2$ & $H_3$ & $H_4$ \\
	\hline  \\[-0.36cm]
	$X_1^2$ & 2 & - & - & - & - \\
$X_2^2$ & 1 & 0 & - & - & - \\
$X_3^2$ & 1 & 0 & 6 & - & - \\
$X_4^2$ & 1 & 0 & 3 & 9 & - \\
$X_5^2$ & 1 & 0 & 6 & 0 & 84 \\
\end{tabular}
\par\bigskip
\begin{tabular}{ c || c | c | c | c  }
   & $H_0$ & $H_1$ & $H_2$ & $H_3$ \\
	\hline \\[-0.36cm]
$X_1^3$ & 3 & - & - & - \\
$X_2^3$ & 1 & 1 & - & - \\
$X_3^3$ & 1 & 0 & 20 & - \\
$X_4^3$ & 1 & 0 & 3 & 103 \\
\end{tabular}
%\par\bigskip 
\quad
\begin{tabular}{ c || c | c | c | c  }
   & $H_0$ & $H_1$ & $H_2$ & $H_3$\\
	\hline \\[-0.36cm]
$X_1^4$ & 4 & - & - & - \\
$X_2^4$ & 1 & 3 & - & - \\
$X_3^4$ & 1 & 0 & 49 & - \\
$X_4^4$ & 1 & 0 & 3 & 426 \\
\end{tabular}
\par\bigskip
\begin{tabular}{ c || c | c | c   }
   & $H_0$ & $H_1$ & $H_2$  \\
	\hline \\[-0.36cm]
$X_1^5$ & 5 & - & -  \\
$X_2^5$ & 1 & 6 & -  \\
$X_3^5$ & 1 & 0 & 99  
\end{tabular}
%\par\bigskip
\
\begin{tabular}{ c || c | c | c  }
   & $H_0$ & $H_1$ & $H_2$  \\
	\hline \\[-0.36cm]
$X_1^6$ & 6 & - & - \\
$X_2^6$ & 1 & 10 & - \\
$X_3^6$ & 1 & 0 & 176 
\end{tabular}
%\par\bigskip
\
\begin{tabular}{ c || c | c | c }
   & $H_0$ & $H_1$ & $H_2$\\
	\hline \\[-0.36cm]
$X_1^7$ & 7 & - & - \\
$X_2^7$ & 1 & 15 & - \\
$X_3^7$ & 1 & 0 & 286
\end{tabular}
\vspace{0.5cm}
\caption{The dimension of the homology groups $H_k(X_s^m;\mathbb{Q})$ for $1\leq m\leq7$ and various $s$.}
\label{t:homdimc}
\end{table}
As in the uncolored case (cf.\ \cite{hv}) the homology groups of $X_s^m$ can be calculated by a computer program (for details see \cite{mmm}). 
Endowing graphs with the additional data of a coloring leads to an even greater growth of the number of cubes with increasing $s$. Consequently, computing the homology of $X_s^m$ by explicit calculation gets more difficult when increasing the number of colors. Thus, the maximal number of external legs $s$ to which the calculations can be performed decreases with $m$.
The results for the homology dimensions for different numbers of colors are listed in Table \ref{t:homdimc}, a specific choice of generators for each group can be found in \cite{mmm}.

\subsection{Special cases} In some cases the usual suspects in the algebraic topologist's toolbox allow for direct derivation of results. 

For $s=1$ the moduli spaces $X_1^m$ simply consists of $m$ points, one for each rose graph $R_{1,1}$ colored by $c \in \set m$. Hence, $H_0(X_1^m;\mb Z)\cong \mb Z^m$ and all other homology groups are trivial. 

Note that in every other case $X_s^m$ is path-connected and therefore $H_0(X_s^m;\mb Z)\cong \mb Z$. For $s=2$ this allows to calculate the first homology group by an Euler characteristic argument. 

\begin{prop}
For $X_2^m$ we have 
\begin{equation*}
 H_1(X_2^m;\mb Z) \cong \mb Z^{ \frac{(m-1)(m-2)}{2} }.
\end{equation*}
\end{prop}

\begin{proof}
 Instead of giving a combinatorial proof, we use a Mayer-Vietoris sequence to provide some insight into the topology at work.
 
Divide the space $X_2^m$ into $X_2^m=\mathring{A}\cup\mathring{B}$ with
\begin{align*}
A&:=\{(G,F,c)\;|\; m\in \mathrm{Im}(c)\} , \\
B&:=\{(G,F,c)\;|\; c(e)\neq m\text{ for at least one }e\in E\}.
\end{align*}
The subspaces $A$ and $B$ are depicted in Figure \ref{fig:MV_A} and Figure \ref{fig:MV_B}, respectively. In these figures, black labels represent the labels of external legs, while red labels represent the coloring.
\begin{figure}[h]
 \centering
   \includegraphics[width=11cm]{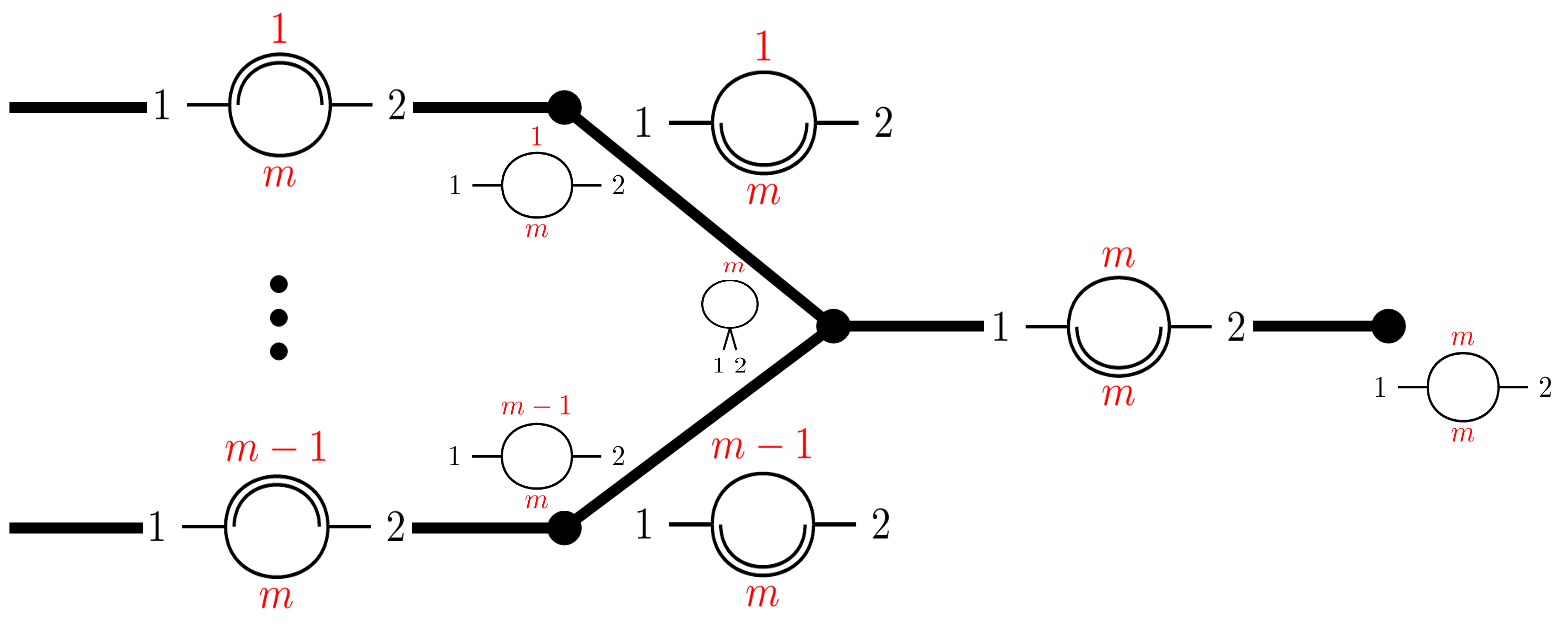}
   \caption{The subspace $A\subset X_2^m$ containing all cubes corresponding to graphs with at least one edge of color $m$.}
  \label{fig:MV_A}
\end{figure}
\begin{figure}[h]
 \centering
\includegraphics[width=11cm]{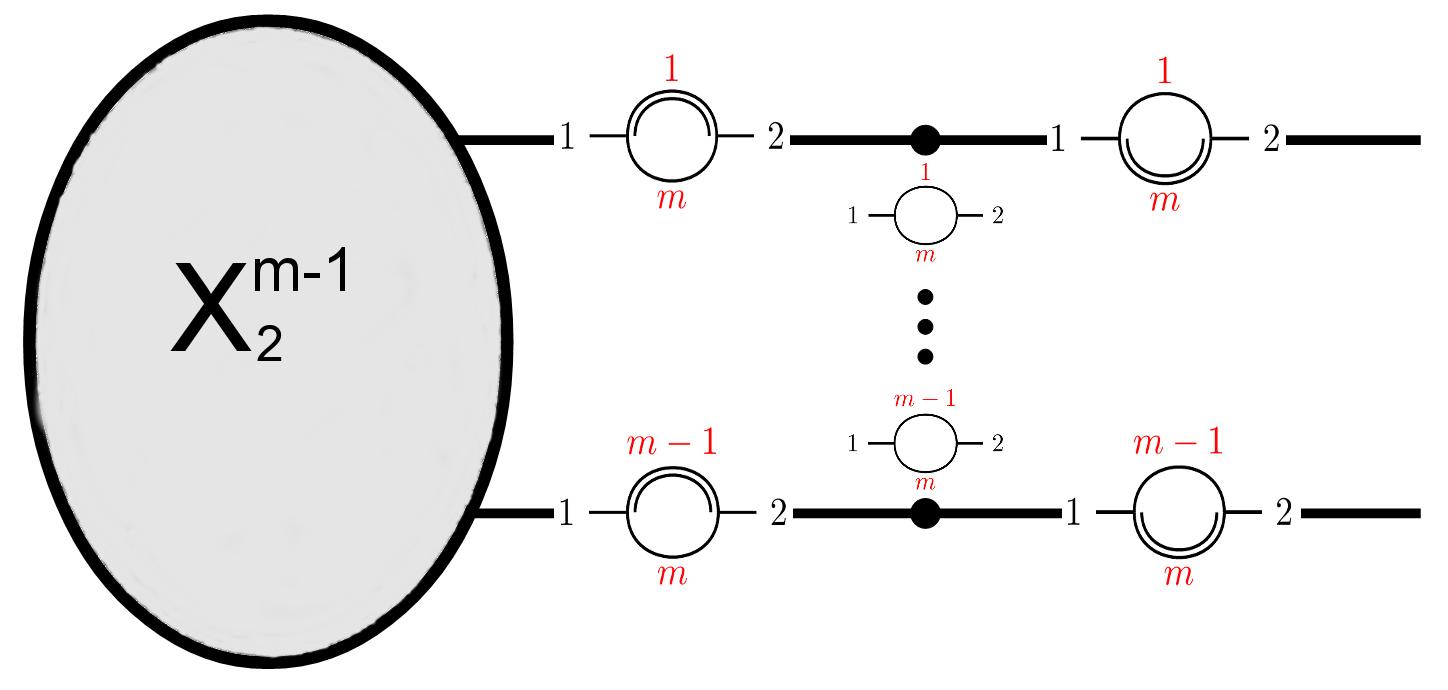}
\caption{The subspace $B\subset X_2^m$ containing all cubes corresponding to graphs with at least one edge not colored by $m$.}
\label{fig:MV_B}
\end{figure}

The intersection $A\cap B$ consists of all cubes corresponding to graphs that contain an edge colored with $m$ but with the other edge colored differently. Figure \ref{fig:MV_AintB} illustrates this intersection.
\begin{figure}[]
 \centering
\includegraphics[width=10cm]{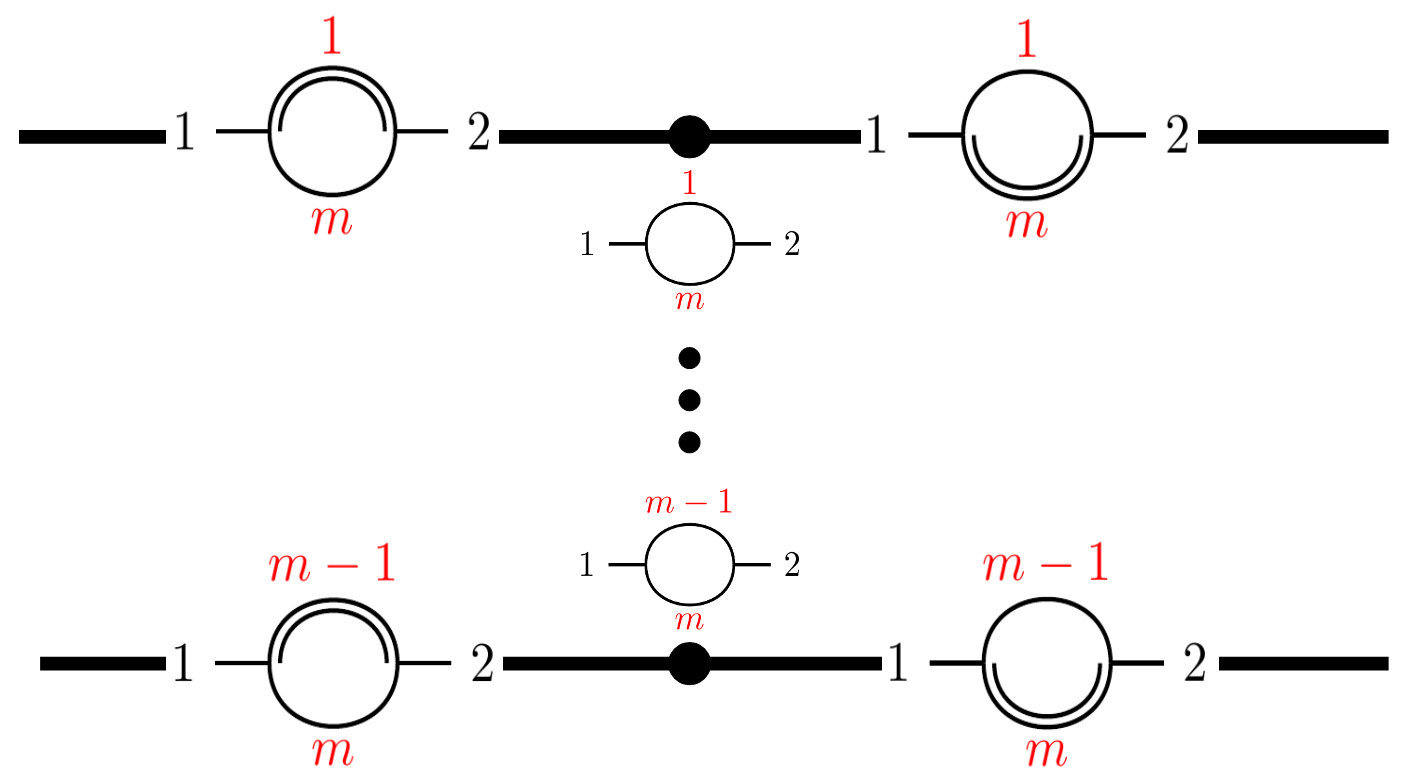}
\caption{The intersection $A\cap B$ of the involved subspaces is a disjoint union of line segments.}
\label{fig:MV_AintB}
\end{figure}
$A$ is contractible and $B$ deformation retracts to a subspace homeomorphic to $X_2^{m-1}$ by shrinking all cubes containing an edge colored by $m$ to zero length. Their intersection $A\cap B$ consists of $m-1$ disjoint lines and can be deformation retracted to $m-1$ disjoint points.

The reduced Mayer-Vietoris sequence for $X_2^m$ and its subspaces $A$ and $B$ reads
\begin{align*}
0 & \longrightarrow \ti {H}_1(A\cap B;\mb Z) \longrightarrow \ti{H}_1(A;\mb Z)\oplus\ti{H}_1(B;\mb Z) \longrightarrow \ti{H}_1(X_2^m;\mb Z) \\
 & \longrightarrow \ti{H}_0(A\cap B;\mb Z) \longrightarrow \ti{H}_0(A;\mb Z)\oplus\ti{H}_0(B;\mb Z)\longrightarrow 0
\end{align*}

Now $\ti{H}_n(A;\mb Z)=0$ since $A$ is homotopy equivalent to a point. By the additivity axiom
\begin{equation*}
\ti{H}_1(A\cap B;\mb Z)\cong\bigoplus_{i=1}^{m-1}\ti H_1(\{*\};\mb Z)=0
\end{equation*}
and
\begin{equation*}
H_0(A\cap B;\mb Z)\cong H_0(\bigsqcup_{i=1}^{m-1}\{*\};\mb Z)\cong\mathbb{Z}^{m-1}.
\end{equation*} 
Hence, $\ti{H}_0(A\cap B;\mb Z)\cong\mathbb{Z}^{m-2}$ and furthermore $\ti{H}_1(B;\mb Z)\cong \ti{H}_1(X_2^{m-1};\mb Z)$. Therefore, the above long exact sequence contains a short exact sequence which after application of the above isomorphisms reads
\begin{equation*}
0 \longrightarrow \ti H_1(X_2^{m-1};\mb Z)\longrightarrow \ti H_1(X_2^m;\mb Z)\longrightarrow \mathbb{Z}^{m-2}\longrightarrow 0,
\end{equation*}
where the first zero is $\ti H_1(A\cap B;\mb Z)$. Since $\mb Z^{m-2}$ is free, this sequence splits and we immediately obtain
\begin{equation*}
\ti H_1(X_2^m;\mb Z)\cong\ti H_1(X_2^{m-1};\mb Z)\oplus\mathbb{Z}^{m-2}.
\end{equation*}
With this the proposition follows by induction.
\end{proof}

We finish this section with a general remark on the calculation of the rational homology and the role of the colors in it. 

Note that $\m {MCG}_{n,s}^m$ naturally contains $m$ copies of $\m {MG}_{n,s}$ as the subspaces corresponding to all graphs whose edges are colored identically. We use this fact by considering the subspace
\begin{equation*}
A := \{(G,\lambda,c)\in \m {MCG}_{n,s}^m \; | \; c \text{ is constant}\}
\end{equation*}
and the long exact sequence of the pair $(\m {MCG}_{n,s}^m,A)$. To simplify notation we set $X:=\m {MCG}_{n,s}^m$. The exact sequence then reads
\begin{equation*}
\ldots \rightarrow H_{k+1}(X,A;\mathbb{Q})\rightarrow H_k(A;\mathbb{Q}) \rightarrow H_k(X;\mathbb{Q}) \rightarrow H_k(X,A;\mathbb{Q})\rightarrow \ldots
\end{equation*}

We have $A\cong\sqcup_{i=1}^m \m {MG}_{n,s}$ and therefore $H_k(A;\mathbb{Q})\cong H_k(\m {MG}_{n,s};\mathbb{Q})^m$ for all $k\in \mathbb{N}_0$. This establishes a close connection between the colored and uncolored case via the relative homology groups $H_k(X,A;\mb Q)$.
\newline

The above ansatz and variations of it - utilizing various filtrations induced by the coloring - may also be used to improve computer calculations. In addition, there are many interesting maps between these moduli space for different $n$, $s$ and $m$ given by attaching self-loops, adding, removing or gluing legs and recoloring edges. The most interesting case surely is the variation of $n$ which will be pursued in future work (cf.\ \cite{chkv} for applications of these ideas in the uncolored case).

\subsection{Homological stability}
In the case of more than two legs $s>2$ the first homology groups of $X_s^m$ stabilize in the sense that $H_1$ is trivial for all numbers of legs and colors. 

\begin{prop}\label{prop:firsthom}
 For all $s \geq 3$ and $m \geq 1$ 
 \begin{equation*}
  H_1(X_s^m;\mb Z)=0.
 \end{equation*}
\end{prop}

\begin{proof}
The statement is a corollary of Theorem \ref{homotopy} which states that all of these spaces are simply connected. 
\end{proof}

The results in Table \ref{t:homdimc} suggest that this stability with respect to $m$ actually holds for all groups $H_i(X_s^m)$ with $i<s-1$.
\begin{conj}\label{conj}
For all $m>1$ and $0\leq i < s-1$ there is an isomorphism 
\begin{equation*}
H_i(X_s^m; \mb Q) \cong H_i(X_s^1; \mb Q) = H_i(\m {MG}_{1,s};\mb Q).
\end{equation*}
\end{conj}

For general rank $n$, consider the cubical chain complexes used to compute the homology groups of $\m {MCG}_{n,s}^m$. For any $m\in\mathbb{N}$ there is a canonical map that forgets the coloring of each graph,
\begin{align*}
 f^{(m)}:C_*(\m {MCG}_{n,s}^m) &\longrightarrow C_*(\m {MG}_{n,s}), \\
 ( G,F,c) &\longmapsto (G,F).
\end{align*}

In the other direction there is a also a map defined by equipping a graph with all possible colorings
\begin{align*}
c^{(m)} : C_*(\m {MG}_{n,s}) & \longrightarrow C_*(\m {MCG}_{n,s}^m), \\
(G,F) & \longmapsto \sum_\text{all colorings c}\frac{1}{m^{|E(G)|}} (G,F,c).
\end{align*}

A short calculation shows that both maps are chain maps. Let $f_\ast^{(m)}$ and $c_\ast^{(m)}$ denote the corresponding induced homomorphism on homology. Then we find
\begin{prop}\label{prop:surj}
$f_\ast^{(m)}\circ c_\ast^{(m)} = id_{H_*(\m {MG}_{n,s};\mathbb{Q})}$, that is $f_\ast^{(m)}$ is surjective and $c_\ast^{(m)}$ is injective.
\end{prop}
\begin{proof}
Let $k\in\mathbb{N}_0$. For all $[x]\in H_k(\m {MG}_{n,s};\mathbb{Q})$ represented by $\sum_iq_i (G_i,F_i)$ with $q_i\in\mathbb{Q}$ we have
\begin{equation*}
(f_\ast^{(m)}\circ c_\ast^{(m)})[x]=\left[ \sum_i q_i(f^{(m)}\circ c^{(m)})( G_i,F_i) \right]
\end{equation*}
 and
\begin{align*}
(f^{(m)}\circ c^{(m)})(G_i,F_i)& =  f^{(m)} \big(\sum_\text{all colorings c}\frac{1}{m^{|E_(G)|}} (G_i,F_i,c) \big)\\
&=\sum_\text{all colorings c}\frac{1}{m^{|E(G)|}} f^{(m)}( G_i,F_i,c ) \\
&=\sum_\text{all colorings c}\frac{1}{m^{|E(G)|}} (G_i,F_i) \\
&=( G_i,F_i ).
\end{align*}
\end{proof}

 This means in particular that $\dim H_k(X_s^1;\mathbb{Q})\leq\dim H_k(X_s^m;\mathbb{Q})$ holds for all $k$, $s$ and $m$, proving ``one half'' of the conjecture.

\subsection{Homotopy type} Now we take a look at the degree of connectivity of the spaces $X_s^m$.
\newline

First consider the case $s=1$. The moduli spaces $X_1^m$ simply consist of $m$ points, one for each rose graph $R_{1,1}$ colored by $c\in \set m$. Hence, $\pi_0(X_1^m)=\{1,\ldots,m \}$. 

All other spaces $X_s^m$ are path-connected, hence also connected. 

\begin{prop}\label{prop:pi1s2}
For $s=2$ and $m\in \mb N$ we have $\pi_1(X_2^m) \cong F_{\frac{1}{2}(m-1)(m-2)}$.
\end{prop}

\begin{proof}
$X_2^m$ is a one-dimensional CW-complex. Its zero-cells are represented by $m$ $i$-colored roses $r_i$ and $m$ ``banana graphs'' (i.e.\ a double edge connecting two vertices) $b_j$, where both edges have length $\frac{1}{2}$ and are colored by a single color $j$. The one-skeleton of $X_2^m$ is formed by arcs $\beta_i$, connecting $b_i$ to $r_i$ by collapsing one of its two edges, and by arcs $\beta_{ij}$, connecting $r_i$ to $r_j$ by either collapsing the $i$-colored or $j$-colored edge (note that in the cubical description of $X_2^m$ the latter arcs correspond to the union of two one-dimensional cubes). The arcs $\beta_i$ are pairwise disjoint, hence contracting them defines a deformation retraction from $X_2^m$ to the complete graph $K_m$ on $m$ vertices (viewed as CW-complex). This, in turn, is homotopy equivalent to a wedge of $\binom{m-1}{2}=\frac{1}{2}(m-1)(m-2)$ circles $S^1$, each one being a generator of $\pi_1(X_2^m)$. 
\end{proof}

Rather surprisingly, in the general case the fundamental group does not recognize the number of colors.

\begin{thm}\label{homotopy}
 For $s \geq 3$ and $m\geq 1$ all the moduli spaces $X_s^m$ are simply connected.
\end{thm}

\begin{rem} 
The theorem should be optimal in the sense that the spaces $X_s^m$ cannot be higher connected. This follows from the \textit{Hurewicz theorem} (see e.g.\ \cite{ah})
together with the results for $H_2(X_s^m;\mb Q)$ in Table \ref{t:homdimc}. In fact, as a corollary we find $\pi_2(X_s^m)\otimes \mb Q \cong H_2(X_s^m;\mb Q).$ Hence $\pi_2(X_s^m)$ also stabilizes for $m\geq 1$.
\end{rem}

For the proof we need a general fact about CW-complexes that allows to compute their fundamental groups by a rather elementary process. 

\begin{prop}[Proposition 6.48 in \cite{switzer}]\label{prop:fugrp}
Let $X$ be a connected CW-complex. Let $i: X^{(1)} \rightarrow X$ denote the inclusion of its one-skeleton and choose a basepoint $x_0 \in X^{(1)}$. Then $i_*: \pi_1(X^{(1)},x_0) \rightarrow \pi_1(X,x_0)$ is an epimorphism and $\ker i_*$ is the normal subgroup generated by the elements $[\varphi_\alpha]$ for $\varphi_\alpha: (S^1,s_0) \rightarrow (X^{(1)},x_0)$ the attaching maps of the two-cells $e^2_\alpha$.
\end{prop}

\begin{proof}(of Theorem \ref{homotopy})
 By Proposition \ref{prop:fugrp} we only need to consider $X:=(X_2^m)^{(2)}$, the two-skeleton of $X_s^m$. It has the following CW-structure: 
\begin{itemize}
 \item Its 0-dimensional cells $r_i$ and $b_i(U,V)$ correspond to 
\begin{itemize}
 \item[-] all roses $R_i$ with their single edge colored by $i \in \set m$,
 \item[-] all banana graphs $B_i(U,V)$ with both edges of length $\frac{1}{2}$ and colored by $i \in \set m$ and $U \sqcup V = \set s$ a partition of the set of legs to the two vertices of $B_i$.
\end{itemize}
 \item Its 1-dimensional cells correspond to 
 \begin{itemize} 
 \item[-] arcs $\beta_i(U,V)$ of length $\frac{1}{2}$ parametrizing banana graphs $B_i(U,V)$ with two edges of unequal length, both colored by $i \in \set m$ and $U \sqcup V = \set s$,
  \item[-] arcs $\beta_{i,j}(U,V)=$ of unit length parametrizing banana graphs $B_{i,j}(U,V)$ with both edges colored differently by $i,j \in \set m$, $i\neq j$, and $U \sqcup V = \set s$.
 \end{itemize}
\item Its 2-dimensional cells $S_{i,j,k}(A,B,C)$ parametrize ``triangle graphs'' (i.e.\ cyclic graphs on three labeled vertices) $T_{i,j,k}(U,V,W)$, colored by $i,j,k \in \set m$, and $U \sqcup V \sqcup W = \set s$ distributing the legs onto the three vertices of $T_{i,j,k}$.
\end{itemize}

The attaching maps can be described as follows. Arcs $\beta_i(U,V)$ connect $r_i$ to $b_i(U,V)$, and arcs $\beta_{i,j}(U,V)$ connect $r_i$ to $r_j$. A two-cell $S_{i,j,k}(A,B,C)$ is attached to three arcs $\beta_{k,j}(A\sqcup B,C)$, $\beta_{i,k}(A,B\sqcup C)$ and $\beta_{i,j}(A\sqcup C,B)$, represented by bananas obtained from $T_{i,j,k}(A,B,C)$ by collapsing a single edge (here $\beta_{i,i}:=\beta_i$ if $i,j,k$ are not all distinct). 

\textbf{Step 1.} Replace $X$ with a simpler homotopy equivalent space.  

Let $Z\subset X$ denote the subcomplex consisting of all $m$ arcs from the rose vertices $r_i$ to the banana vertices $b_i(U,V)$ and $m-1$ arcs from $r_i$ to $r_{i+1}$ defined by varying the edge lengths in $B_{i,i+1}(U_i,V_i)$ (here $i\in \{1,\ldots,m-1\}$ and $U_i\sqcup V_i= \set s$ freely chosen partitions). 
Note that $Z$ forms a maximal tree in the one-skeleton of $X$. In particular, all zero-dimensional cells of $X$ lie in $Z$. Therefore, in the process of collapsing this subcomplex all these cells get identified to a single vertex $y_0$. Since $Z$ is contractible, the quotient map $\pi: X \rightarrow Y:=X/Z$ is a homotopy equivalence (see e.g.\ \cite{ah}).

The quotient $Y$ has the following CW-structure:
\begin{itemize}
 \item It has a single vertex $y_0$. 
 \item Its 1-dimensional cells $\beta_{i,j}(U,V)$ are arcs describing bananas $B_{i,j}(U,V)$ with both their edges colored differently by $i,j \in \set m$, $i\neq j$, and $U \sqcup V = \set s$. If $j=i+1$, then $U\sqcup V$ is any partition except the special choice from the definition of $Z$, i.e.\ $U\neq U_i$, $V\neq V_i$. Note that every arc describes a loop in $Y$, based at $y_0$.
\item Its 2-dimensional cells are disks $S_{i,j,k}(A,B,C)$ parametrizing triangle graphs $T_{i,j,k}(A,B,C)$ where $i,j,k \in \set m$ and $A \sqcup B \sqcup C = \set s$. The attaching maps of these disks are described below. In the special case $i=j=k$ we have $\partial S_{i,i,i}(A,B,C)=\{ y_0 \}$, so $S_{i,i,i}(A,B,C)$ is a sphere. 
\end{itemize}
 
\textbf{Step 2.} Calculate $\pi_1(Y,y_0)$.

Since the one-skeleton of $Y$ is a wedge of circles, $\pi_1(Y,y_0)$ is generated by all arcs $\beta_{i,j}$ described above, subject to relations induced by the attaching maps of two-cells. We now show that these relations kill all generators.

\begin{figure}[!htb]
 \centering
 \includegraphics[width=13cm]{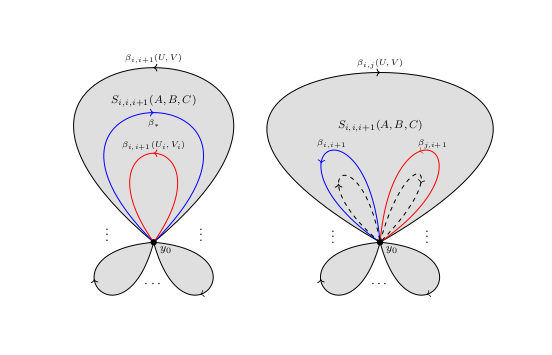}
\caption{Arcs in the two-skeleton of $Y$.}
\label{fig:twoskelet}
 \end{figure}

Let $\beta_{i,j}(U,V)$ be such a generator. First we consider the case $j=i+1$ and assume $U \cap U_i \neq \emptyset$ (otherwise we permute $U$ and $V$).
The arc $\beta_{i,i+1}(U,V)$ is one boundary component of the disk $S_{i,i,i+1}(A,B,C)$ with 
\begin{equation}
 \partial S_{i,i,i+1}(A,B,C)= \beta_{i,i+1}(A\sqcup B,C) \cup \beta_{i,i+1}(A,B\sqcup C) 
\end{equation}
such that $A=U$ and $B \sqcup C=V$.
We choose 
\begin{equation*}
   B= U_i \cap V, \quad  C=  ( U_i \setminus U  ) \triangle V,
\end{equation*}
 where $\triangle$ denotes the symmetric difference of sets. Note that $C$ is also given by $\set{s} \setminus (U \cup (U_i \cap V))$. Then the other boundary component of $S_{i,i,i+1}(A,B,C)$ is the arc 
\begin{equation*}
\beta_* := \beta_{i,i+1}\big( U\cup (U_i \cap V), \set s \setminus (U \cup (U_i \cap V)) \big).
\end{equation*}
If $ U\subset U_i $, this is the arc $\beta_{i,i+1}(U_i,V_i)$ which has been collapsed in the construction of $Y$. In other words, the disk $S_{i,i,i+1}(A,B,C)$ has only $\beta_{i,i+1}(U,V)$ as boundary. Thus, $\beta_{i,i+1}(U,V)$ is homotopic to the constant loop, 
\begin{equation*}
 [\beta_{i,i+1}(U,V)]=[y_0]= e \in \pi_1(Y,y_0).
\end{equation*}
If not, then the arc $\beta_*$ lies in the boundary of another disk $S_{i,i,i+1}(M,N,R)$ with
\begin{equation*}
  \partial S_{i,i,i+1}(M,N,R) = \beta_{i,i+1}(M\sqcup N, R) \cup \beta_{i,i+1}(M,R\sqcup N)
\end{equation*}
for 
\begin{align*}
& M= \set s \setminus (U \cup (U_i\cap V)) \\
& R= \big( U \cap (U_i \cap V) \big) \cap U_i = U_i  \\
&  S= \set s \setminus ( M \sqcup R ).
\end{align*}
Then we find $\partial S_{i,i,i+1}(M,N,R)$ to be given by
\begin{equation*}
   \beta_{i,i+1}(U_i, V_i) \cup \beta_{i,i+1} \big( U \cup (U_i \cap V) , \set s  \setminus   (U \cup (U_i\cap V)) \big)
\end{equation*}
and therefore $[\beta_{i,i+1}(U,V)]=[\beta_*]=[y_0]= e$. This is depicted in the left picture in Figure \ref{fig:twoskelet}.

Algebraically, this sequence of relations reads
\begin{equation*}
 e =  \beta_*^{-1} \cdot \beta_{i,i+1}(U,V), \quad  e = \beta_{i,i+1}(U_i,V_i)^{-1} \cdot \beta_*
\end{equation*}
and since $\beta_{i,i+1}(U_i,V_i)=e$ this shows $\beta_{i,i+1}(U,V)=e$.

Now suppose $i \neq j-1$, say $i< j-1$. In this case, $\beta_{i,j}(U,V)$ is the boundary component of a disk $S_{i,j,k}(A,B,C)$ for $k=i+1$, $A\sqcup C=U$ and $B=V$,
\begin{align*}
 \partial S_{i,j,k}(A,B,C) &= \beta_{j,k}(A \sqcup B,C) \cup \beta_{i,k}(A,B\sqcup C) \cup \beta_{i,j}(A\sqcup C,B) \\
 & = \beta_{j,i+1}(A \sqcup B,C) \cup \beta_{i,i+1}(A,B\sqcup C) \cup 
 \beta_{i,j}(U,V).
\end{align*}
By the previous arguments we know that the loop $\beta_{i,i+1}(A,B\sqcup C)$ is trivial in $\pi_1(Y,y_0)$. Now repeat this step for the other arc $\beta_{j,i+1}(A \sqcup B,C)=:\beta_{i+1,j}(M,N)$, i.e.\ choose a disk $S_{i+1,j,i+2}(I,J,K)$ with $I \sqcup K=M$ and $J=N$, such that
\begin{equation*}
 \partial S_{i+1,j,i+2}(I,J,K) = \beta_{j,i+2}(I \sqcup J,K) \cup  \beta_{i+1,i+2}(I,J \sqcup K) \cup \beta_{i+1,j}(M,N).
\end{equation*}
Again, $\beta_{i+1,i+2}(I,J \sqcup K)$ is trivial, so this equation expresses $\beta_{i+1,j}(M,N)$ by $\beta_{i+2,j}(I \sqcup J,K)$. Repeat this process until after $j-(i+1)$ steps all arcs are revealed to be trivial, so that we conclude $\beta_{i,j}(U,V)= e$.
\end{proof}
 
\begin{rem}
 This method of proof does not work for the moduli spaces of holocolored graphs (indeed, $\m {MHG}_{1,3}$ and $\m {MHG}_{1,3}$ are homeomorphic to the two-dimensional torus $T^2$ which has $\mb Z^2$ as fundamental group). Roughly speaking, although $X_s^m$ seems to be more complicated at first sight, the additional ``stuff'' allows to simplify the relations in the presentation of $\pi_1(X_m^s)$ as we have done above.
\end{rem}

\subsection{Euler characteristic}
The Euler characteristic of $X_s^m$ can be computed by simply counting all admissible colored graphs representing cells in $Y \simeq X_m^s$, the CW-complex described in the proof of Theorem \ref{homotopy}.

\begin{prop}\label{prop:ecm}
 The Euler characteristic $\chi(X_m^s)$ is given by 
 \begin{equation*}
\chi(X_m^s)= 2^{s-1} \frac{m(1-m)}{2} + \frac{m(m+1)}{2} + \sum_{k=3}^s (-1)^{k-1} \stirling{s}{k} \frac{(k-1)!}{2} m^k.
 \end{equation*}
\end{prop}

\begin{proof}
 Recall the notation from the proof of Theorem \ref{homotopy}. The CW-complex $Y$ has a single zero-dimensional cell, represented by a rose graph with $s$ thorns $R_{1,s}$. One-dimensional cells are represented by banana graphs $B_{i,j}(U,V)$ with $i\neq j \in \{ 1, \ldots, m\}$ and $U\sqcup V = \{1,\ldots,s \}$, except the especially chosen elements $B_{i,i+1}(U_i,V_i)$. Hence, there are
 \begin{equation*}
 \stirling{s}{2} \frac{m(m-1)}{2} - (m-1)                                                                                                                                                                                                                                                                                                                                                                                                                                                                                                                                                                       \end{equation*}
 different representatives. The same argument gives the number of cells of dimension $k-1$. It is the number of leg partitions times non-equivalent vertex configurations times colorings of a graph with $k$ edges. This number is given by the formula
 \begin{equation*}
  \stirling{s}{k} \frac{(k-1)!}{2} m^k
 \end{equation*}
and therefore
\begin{align*}
 \chi(X_m^s) = & m -  \stirling{s}{2} \frac{m(m-1)}{2}  + \sum_{k=3}^s (-1)^{k-1} \stirling{s}{k} \frac{(k-1)!}{2} m^k \\
  = & 2^{s-1}\frac{m(1-m)}{2} + \frac{m(m+1)}{2} + \sum_{k=3}^s (-1)^{k-1} \stirling{s}{k} \frac{(k-1)!}{2} m^k.
\end{align*}
\end{proof}

For $m=1$ this simplifies considerably.
\begin{prop}
 The sequence $\chi(s):= \chi (\m MG_{1,s})$ is given by
 \begin{equation*}
  \chi(1)=1,\ \chi(s) = 2^{s-2} \quad  \text{ or } \quad \chi(1)=1, \ \chi(s) = \sum_{k=1}^{s-1} \chi(k)
 \end{equation*}
with generating function $G(t)=t (1-t) ( 1-2t )^{-1}$.
\end{prop}

\begin{proof}
For $s>2$ a long, but straightforward calculation, using $\stirling{s+1}{k}=k\stirling{s}{k} + \stirling{s}{k-1}$, shows that $\chi(s+1)-\chi(s)=\chi(s)$ from which both formulae follow. For $G(t)=\sum_{s=1}\chi(s)t^s$ we calculate
 \begin{align*}
   tG & =  \sum_{s=1}\chi(s)t^{s+1} = \chi(1)t^2 + \chi(2)t^3 + \ldots \\
    & = \frac{\chi(2)}{2}t^2 + \frac{\chi(2)}{2}t^2 + \frac{\chi(3)}{2}t^3 + \ldots \\
    & = \frac{ G -\chi(1)t }{2} + \frac{\chi(2)}{2}t^2 
 \end{align*}
and therefore $G(t)=t (1-t) ( 1-2t )^{-1}$.
\end{proof}

\subsection{Highest non-trivial Betti number}

Let $b_s(m)$ denote the highest rank Betti number $h_{s-1}(X_s^m)= \dim H_{s-1}(X_s^m;\mb Q)$ which, if Conjecture \ref{conj} holds, is the only Betti number of $X_s^m$ depending on $m$. Above we have seen that $b_1(m)=h_0(X_1^m)=m$ and $b_2(m)=\frac{1}{2}(m-1)(m-2)$. Moreover, the results in Table \ref{t:homdimc} suggest that $b_s(m)$ grows polynomially in $m$.

\begin{thm}\label{t:bettipoly}
The Betti number $b_s(m)=h_{s-1}(X_s^m)$ grows at most polynomially in $m$ of degree $s$.
\end{thm}

\begin{proof}
Let $s\geq 3$. By Proposition \ref{prop:ecm} the Euler characteristic $\chi(X_s^m)$ is a polynomial in $m$ of degree $s$.  Moreover, the number of $(k-1)$-dimensional cells in $X_s^m$ is $\stirling{s}{k}\frac{(k-1)!}{2} m^k$, so that 
\begin{equation*}
h_{k-1}(X_s^m)\leq \dim C_{k-1}(X_s^m;\mb Q) = \stirling{s}{k}\frac{(k-1)!}{2} m^k
\end{equation*} 
bounds the $(k-1)$-th Betti number from above by a polynomial of degree $k$. 

This shows that $b_s(m)$ is bounded by a polynomial of degree $s$, except if $\chi(X_s^m)= 0$ which may happen for at most finitely many $m$.\footnote{Furthermore, the \textit{integer root theorem} implies $b_s(m)\neq 0$ for all even $m \in 2\mb Z\setminus\{0\}$.}
\end{proof}

\begin{rem}
If Conjecture \ref{conj} holds, this growth estimate is sharp, i.e.\ $b_s$ is indeed a polynomial in $m$ of degree $s$. It could then be determined explicitly, either by computing $\chi(X_s^m)$ and the uncolored homology groups from \eqref{eq:homouter} or by interpolating the values $b_s(m)$ for $m=1,\ldots,s+1$. 
\end{rem}

\section{Holocolored graphs}\label{s:hc}

In contrast to the case of arbitrarily colored graphs the spaces $\m{MHG}_{1,s}$ and $\m {MRG}_{1,s}$ can be directly (without subdivision of cells) considered as $\Delta$-complexes. Therefore, we may use instead of the cubical chain complex the associated semi-simplicial chain complex which has fewer generators. We thereby obtain a \textit{graph complex}, somewhat in the spirit of Kontsevich's \textit{graph homology} \cite{ko1,ko2}, that computes the homology of the moduli spaces of holocolored graphs.

\subsection{The moduli spaces $\m{MHG}_{1,s} $}
Let $\ti X_s$ denote the space $\m{MHG}_{1,s}$.  
In the semi-simplicial chain complex for $\ti X_s$ the $k$-dimensional chain groups are generated by simplices $\Delta_{(G,c)}$, one for each isomorphism class of admissible colored graphs on $k+1$ edges. As in the cubical case we abbreviate these generators by $(G,c)$. The boundary operator of this chain complex contains the terms with shrunken edges only and reads\footnote{Note that this simple definition works only for the case $n=1$. Otherwise it might not be possible to shrink an edge to zero length, owing to the occurrence of missing faces.}
\begin{equation*}
\partial_k (G, c):=\sum_{i=1}^{k}(-1)^{i-1} (G_{e_i}, c_{e_i}).
\end{equation*}

Recall that the number of allowed colors is $3(n-1)+s$ which in the one loop case equals $s$, the maximal number of internal edges of the graphs representing elements in $\ti X_s$. 

\begin{figure}[h!]
  \centering
	\includegraphics[width=11cm]{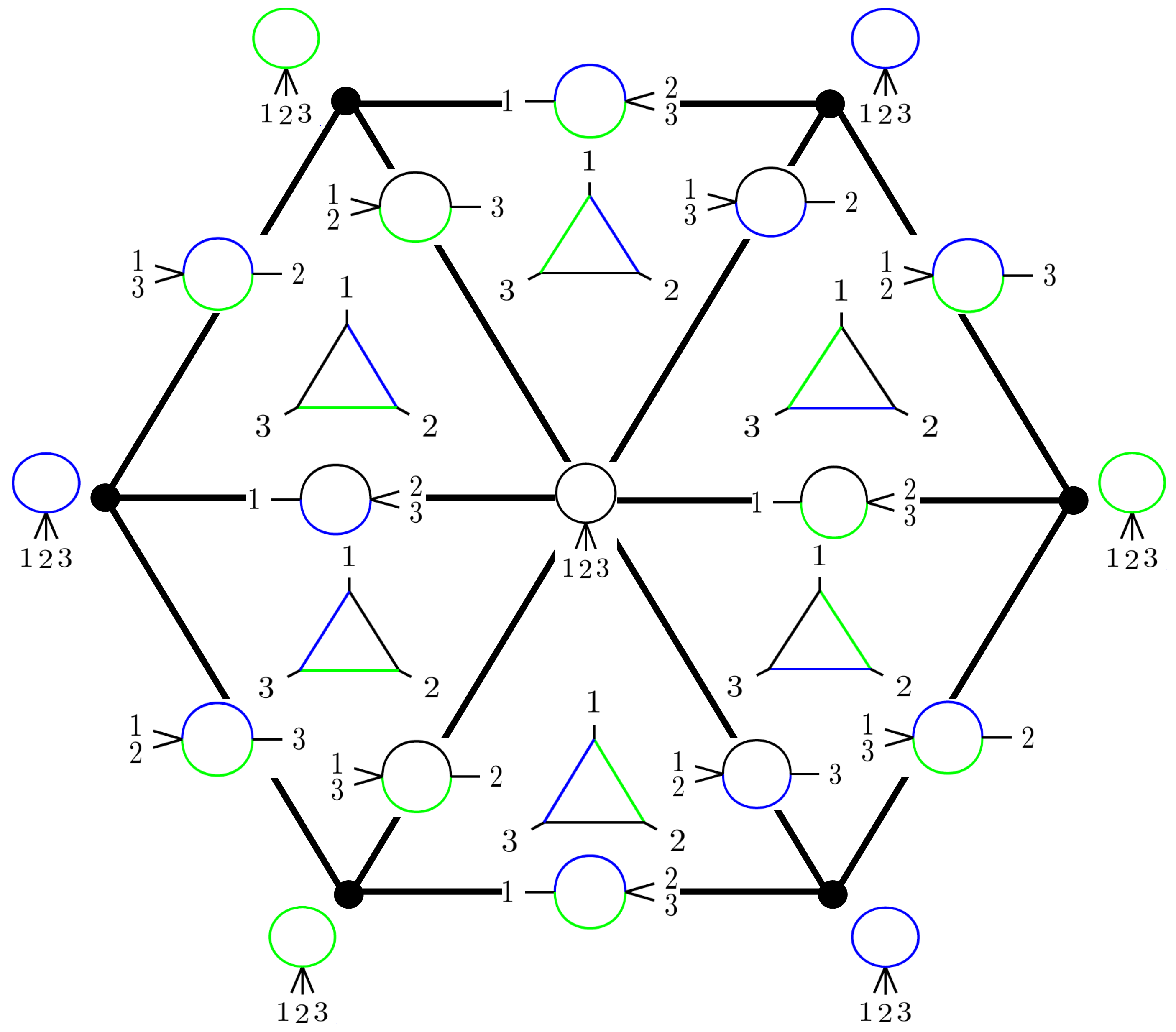}
	\caption{A geometric representation of $\ti {X}_3$ as a $\Delta$-complex.}
	\label{fig:s3r}
\end{figure}

\begin{example}
While for $s=1,2$ the space $\ti X_s$ is just a single point or a single edge, respectively, the space $\ti X_3$ is homeomorphic to the two-dimensional torus $T^2$. Its $\Delta$-complex structure consists of six 2-simplices, nine 1-simplices and 3 vertices with identifications as depicted in Figure \ref{fig:s3r}.
\end{example}

As in the previous case the homology of the moduli space of holocolored graphs can be calculated with the help of a computer program. 
The dimensions of the homology groups of the spaces $\ti X_s$ that could be obtained are listed in Table \ref{t:hom_dirk} for $1\leq s\leq 5$. A particular choice of generators for these groups can be found in \cite{mmm}.

\begin{table}[h]
\centering
\begin{tabular}{ c || c | c | c | c | c }
   & $H_0$ & $H_1$ & $H_2$ & $H_3$ & $H_4$\\
	\hline\\[-1em]
$\ti X_1$ & 1 & - & - & - & - \\
$\ti X_2$ & 1 & 0 & - & - & - \\
$\ti X_3$ & 1 & 2 & 1 & - & - \\
$\ti X_4$ & 1 & 0 & 36 & 3 & - \\
$\ti X_5$ & 1 & 0 & 6 & 824 & 12 \\
\end{tabular}
\caption{The dimension of the homology groups $H_k(\ti X_s;\mb Q)$ for $1\leq s\leq5$ and $0\leq k\leq4$.}
\label{t:hom_dirk}
\end{table}

\subsubsection{Homology in the highest non-trivial dimension} The groups $H_{s-1}(\ti X_s;\mb Z)$ can be determined explicitly by merging the simplices in the $\Delta$-decomposition of $\ti X_s$ into larger cubes. 

\begin{prop}
For $s\geq3$ we have 
\begin{equation}\label{eq:homrainbow}
 H_{s-1}(\ti X_s; \mb Z)\cong\mathbb{Z}^\frac{(s-1)!}{2}.
\end{equation}
\end{prop}

\begin{proof}
We describe a cubical chain complex that computes $H_{s-1}(\ti X_s;\mb Z)$.

A point $(G,\lambda,c) \in \ti X_s$ can be thought of as embedded in $\mathbb{R}^2$ by drawing a circle of finite radius (which is $\frac{1}{2\pi}$ for normalized graphs) to represent the union of all of its edges. Fix an arbitrary point on the circle and identify it with a leg/vertex $v_0\in V(G)$. The length of each edge is uniquely determined by fixing $s-1$ angles $\theta_1,\ldots,\theta_{s-1}$ representing the position of the legs/vertices $v_1,\ldots,v_{s-1}\in V(G)$ on the circle with respect to the position of the distinguished vertex $v_0$. For convenience we choose a parametrization such that angles $\theta_i$ reach from 0 to 1.

In this way $(G,\lambda,c)$ can be represented by a tuple $(\theta_1,\ldots,\theta_{s-1},\sigma)$ where $\sigma\in \Sigma_s$ encodes the coloring $c$. This representation is not unique, since we have
\begin{equation*}
(\theta_1,\ldots,\theta_{s-1},\sigma) \sim (\theta_{s-1},\ldots,\theta_1,\sigma^-)
\end{equation*}
where $\sigma^-$ denotes the permutation $\bigl(\begin{smallmatrix}1 & 2 & \ldots & s\\\sigma_1 & \sigma_s & \ldots &  \sigma_2 \end{smallmatrix}\bigr)$.

This defines a cubical complex in which each $[\sigma] \in \Sigma_s / \mb{Z}_2$ designates an $s-1$ dimensional cube $w_{[\sigma]}:=(\theta_1,\ldots,\theta_{s-1},[\sigma])$. On the associated chain complex the cubical boundary operator $\partial=\partial^+ + \partial^-$ is then given by
\begin{equation*}
\partial^+ w_{[\sigma]}=\sum_{i=1}^{s-1}(-1)^{i-1}(\theta_1,\ldots,\theta_{i-1},0,\theta_{i+1},\ldots,\theta_{s-1},[\sigma])
\end{equation*}
and
\begin{equation*}
\partial^- w_{[\sigma]} =\sum_{i=1}^{s-1}(-1)^i (\theta_1,\ldots,\theta_{i-1},1,\theta_{i+1},\ldots,\theta_{s-1},[\sigma]).
\end{equation*}

To describe the face relations in this complex we define a ``shuffling'' operator $\tau_+: \Sigma_s \rightarrow \Sigma_s$ by
\begin{equation*}
\tau_+ \sigma := \tau_+ \bigl(\begin{smallmatrix}1 & 2 & \cdots & s \\\sigma_1 & \sigma_2 & \cdots &  \sigma_s\end{smallmatrix}\bigr)= \bigl(\begin{smallmatrix}1 & 2 & \cdots & s-1 & s \\\sigma_2 & \sigma_3 & \cdots &  \sigma_s & \sigma_1\end{smallmatrix}\bigr).
\end{equation*}

Letting any external leg rotate once around the graph leads to a cyclic permutation of the edges. More precisely, for any $[\sigma]\in \Sigma_s / \mb Z_2$ we have
\begin{equation*}
(\theta_1,\ldots,\theta_{i-1},0,\theta_{i+1},\ldots,\theta_{s-1},[\sigma])=(\theta_1,\ldots,\theta_{i-1},1,\theta_{i+1},\ldots,\theta_{s-1},[\tau_+\sigma]).
\end{equation*}
 This immediately yields
\begin{equation}\label{eq:cubesrainbow}
\partial_{s-1} w_{[\sigma]}  = \partial_{s-1}^+ w_{[\sigma]}  + \partial_{s-1}^- w_{[\sigma]} = \partial_{s-1}^+ w_{[\sigma]} - \partial_{s-1}^+w_{[\tau_+\sigma]}.
\end{equation}

The highest non-trivial homology group $H_{s-1}(\ti X_s;\mb Z)$ is just $\ker\partial_{s-1}$ and equation \eqref{eq:cubesrainbow} sets up a linear system of equations to determine this kernel. To make this explicit, we quotient out the cyclic permutations generated by $\tau_+$. The quotient group is thus $(\Sigma_s/\mb Z_2)/C_s\cong \Sigma_s/(C_s\ltimes\mathbb{Z}_2)\cong \Sigma_S/D_s$ for $D_s$ the dihedral group. It consists of $n(s):=\frac{s!}{2s}=\frac{(s-1)!}{2}$ equivalence classes and we choose a representative $\sigma_1,\ldots,\sigma_{n(s)}$ for each one.

The matrix representation of $\partial_{s-1}$ with respect to the basis
\begin{equation*}
\left\{ \partial_{s-1}^+w_{[\sigma_i]},\partial_{s-1}^+w_{[\tau_+\sigma_i]},\ldots,\partial_{s-1}^+w_{[\tau_+^{s-1}\sigma_i]} \right\}_{i=1, \ldots, n(s)}
\end{equation*}
of the target space $C_{s-2}(\ti X_s)$ then reads
\begin{equation*}
A_{\partial_{s-1}}=\left(\begin{matrix} 
A 		& 			& 			& \\
		& A 		& 			& \\
		&   		& \ddots	& \\
 		& 			& 			& A
\end{matrix}\right),
\end{equation*}
containing $n(s)$ copies of
\begin{equation*}
A=\left(\begin{matrix} 
1 		& 0			& 0			& \cdots	& 0			& -1\\
-1		& 1 		& 0			& \cdots	& 0			& 0\\
0		& -1 		& 1			& \cdots	& 0			& 0\\
\vdots	& \vdots 	& \vdots	& \ddots	& \vdots	& \vdots\\
0 		& 0			& 0			& \cdots	& 1			& 0\\
0		& 0			& 0			& \cdots	& -1		& 1
\end{matrix}\right).
\end{equation*}

The matrix $A$ can easily be brought into row-echolon form and reveals its rank to be $\mathrm{rank}(A)=s-1$. Thus, the space of solutions of the homogeneous system defined by $A_{\partial_{s-1}}$ is $\frac{(s-1)!}{2}$-dimensional and we obtain
\begin{equation*}
H_{s-1}(\ti{X}_{1,s};\mb Z)\cong\mathbb{Z}^\frac{(s-1)!}{2}.
\end{equation*}
\end{proof}

\subsubsection{Permuting colors} 
Let $\Sigma_C:=\mathrm{Perm}(C)\cong \Sigma_{3(n-1)+s}$ denote the group of permutations of the set of colors $C$. It acts on $\m {MHG}_{n,s}$ by changing the coloring, for $g \in \Sigma_C$ by $g.(G,\lambda,c):=(G,\lambda,g \circ c)$.

This action respects the ``relative semi-simplicial structure'' by mapping open $k$-simplices to open $k$-simplices. Furthermore, it does so transitively on each set of open $k$-simplices in $\m {MCG}_{n,s}$. 
The stabilizer of a point $(G,\lambda,c)$ depends only (up to isomorphism) on the number of edges of its representing graph $G$; if $G$ has $k$ edges we find its stabilizer as the set of permutations that act only on the colors not in the image of $c$, 
\begin{equation*}
\forall x=(G,\lambda,c) \in \m {MCG}_{n,s}: \ {\Sigma_C}_x \cong \mathrm{Perm}(C \setminus \mathrm{im}(c))\cong \Sigma_{3(n-1)+s-k}.
\end{equation*}

This shows that the orbit space of this action is the moduli space of uncolored graphs $\m {MG}_{n,s}$ and the projection $\pi:  \m {MCG}_{n,s} \rightarrow \m {MCG}_{n,s}  / \Sigma_C = \m{MG}_{n,s}$ is a \textit{branched covering map}, i.e.\ a covering map outside a nowhere-dense set, in this case outside of $B=\{ (G,\lambda,c) \in \m {MCG}_{n,s} \mid |E(G)|< 3(n-1)+s \}$. 

For $n=1$ this means that if we decompose $\ti X_s$ into a simplicial complex, for instance by performing two barycentric subdivisions, then the action of $\Sigma_C \cong \Sigma_s$ is simplicial, i.e.\ for every $g$ in $\Sigma_C$ the map $v \mapsto g.v$ is \textit{simplicial}.\footnote{A map between two simplicial complexes $f:K \rightarrow K'$ is \textit{simplicial} if it sends every simplex in $K$ to a simplex in $K'$ by a map taking vertices to vertices.}

In that case the projection $\pi: \ti X_s \rightarrow  \ti X_s / \Sigma_C  = \m{MG}_{1,s}$ that forgets the coloring is a simplicial branched covering map, i.e.\ a simplicial covering map outside the nowhere-dense set $\ti X^{(s-2)}_s$, the $(s-2)$-skeleton of $\ti X_s$. 

Moreover, we have in complete analogy to Proposition \ref{prop:surj}
\begin{prop}\label{prop:projhc}
 The quotient map $\pi: \m {MHG}_{n,s} \rightarrow  \m{MG}_{n,s} $ that forgets the coloring of edges induces a surjection $\pi_*$ on homology.
\end{prop}

\begin{proof}
On the cubical level the right-inverse to $\pi_*$ is induced by the map 
\begin{equation*}
 i:C_*(\m{MG}_{n,s}) \longrightarrow C_*(\m {MHG}_{n,s}), \ (G,F) \longmapsto \sum_{\text{all holocolorings $c$}} \frac{1}{\kappa_G}(G,F,c)
\end{equation*}
for $\kappa_G=\binom{|C|}{|E(G)|}$ the number of holocolorings of $G$.
\end{proof}

\subsubsection{The Euler characteristic of $\ti X_{s}$} The number of $k$-simplices in the moduli space of holocolored graphs can be determined by combinatorial means. 

For an admissible graph $G$ with $k$ internal edges, denote the set of all holocolorings (with $s$ available colors) of $G$ by $\mathcal{C}_{k,s}(G)$, and the set of its non-equivalent external leg structures with $s$ legs by $\mathcal{S}_{k,s}(G)$.

For an one-loop holocolored graph $G$ with $k$ internal edges there are $\binom{s}{k}$ ways to choose which colors the edges of $G$ can have. Furthermore, there are $k!$ ways to permute these edges to get different graphs, except in the case $k=2$ where both permutations of edges yield the same graph. Thus, 
\begin{equation}\label{eq:numcol}
|\mathcal{C}_{k,s}(G)|=\binom{s}{k}\frac{k!}{1+\delta_{k,2}}=\frac{s!}{(s-k)!(1+\delta_{k,2})}.
\end{equation}

The leg structure is uniquely determined by a partition of the $s$ legs into $k$ non-empty groups and an element of $S_k/(C_k\ltimes\mathbb{Z}_2)$ which represents their ordering. There are $\stirling{s}{k}$ number of ways to choose such a partition where $\stirling{\cdot}{\cdot}$ denotes the Stirling number of the second kind \cite{handbook},
\begin{equation*}
\stirling{n}{k}=\frac{1}{k!}\sum_{j=0}^{k}(-1)^{k-j}\binom{k}{j}j^n.
\end{equation*}

For $k\geq 3$ the group $S_k/(C_k\ltimes\mathbb{Z}_2)$ has $\frac{k!}{2}\frac{1}{k}=\frac{(k-1)!}{2}$ elements. Thus, there are exactly this many ways to order the $s$ legs to yield distinct graphs. For $k=1$ and $k=2$ there is clearly only one way to organize the legs. Hence,
\begin{equation}\label{eq:numleg}
|\mathcal{S}_{k,s}(G)|=\stirling{s}{k}\frac{(k-1)!}{2}(1+\delta_{k,1}+\delta_{k,2}).
\end{equation}

With this the number of ($k-1$)-simplices $N_{k,s}$ in $\ti X_s$ can be calculated. Each such ($k-1$)-simplex belongs to a one-loop holocolored graph with $k$ edges and $s$ legs. There are $|\mathcal{C}_{k,s}(G)|\cdot|\mathcal{S}_{k,s}(G)|$ such graphs since the choices of coloring and leg structure are independent, so by \eqref{eq:numcol} and \eqref{eq:numleg}
\begin{equation*}
N_{k,s}	=\binom{s}{k}k!\stirling{s}{k}\frac{(k-1)!}{2}(1+\delta_{k,1})=\frac{(k-1)!}{2-\delta_{k,1}}\binom{s}{k}\sum_{j=0}^{k}(-1)^{k-j}\binom{k}{j}j^s.
\end{equation*}

This in turn can be used to calculate the Euler characteristic of $\ti X_s$,
\begin{align*}
\chi(\ti X_s) &=\sum_{k=1}^{s}(-1)^{k-1}N_{k,s}\\
						&=\sum_{k=1}^{s}(-1)^{k-1}\frac{(k-1)!}{2-\delta_{k,1}}\binom{s}{k}\sum_{j=0}^{k}(-1)^{k-j}\binom{k}{j}j^s\\
						&=\sum_{k=1}^{s}\sum_{j=0}^{k}(-1)^{j+1}\frac{j^s(k-1)!}{2-\delta_{k,1}}\binom{s}{k}\binom{k}{j}.
\end{align*}

For $1\leq s\leq 8$ the Euler characteristic obtained by this formula can be found in Table \ref{t:euler}.
\begin{table}
\centering
\begin{tabular}{ c || c | c | c | c | c | c | c | c  }
$s$ 					& 1 & 2 & 3 & 4 & 5 & 6 & 7 & 8 \\
\hline\\[-1em]
$\chi(\ti X_s)$ & 1 & 1 & 0 & 34 & -805 & 26541 & -1122506 & 59485588 \\
\end{tabular}
\caption{The Euler characteristic of $\ti X_s$ for $1\leq s\leq 8$}
\label{t:euler}
\end{table}

\subsection{Graphs with remembered edges}\label{ss:re}
When considering Feynman graphs in the operator product expansion \cite{lh75,iz}, an edge collapsed to zero length still carries the physical information assigned to it (for instance, its mass or particle type). More precisely, the vertex it gets identified with describes a new type of interaction, depending on the type of the contracted edge. Thus, from a physics perspective it is worthwhile to consider an alternative complex in which face relations respect this restriction. 

Recall the definition of $\m {MRG}_{n,s}$ where again $n$ is the number of loops, $s$ the number of external legs and the number of colors is taken to be the maximal number of internal edges that can occur, $|C|=3(n-1)+s$. Note that we consider the same ground set of holocolored graphs but identify points with respect to a different relation $\sim_*$ given by
\begin{equation*}
(G,\lambda,c) \sim_* (G',\lambda',c') \Longleftrightarrow \exists \varphi:G \longrightarrow G' \text{ homothety s.t.\ } c= c' \circ \varphi.
\end{equation*}

\subsubsection{The homology of $\m {MRG}_{1,s}$}
In the following let $n=1$ and denote by $\bar X_s$ the space $\m {MRG}_{1,s}$. Up to $s=3$ there is no difference between this space and the moduli space of holocolored graphs with ``forgetful edges''. A difference can only occur if there are at least two vertices to which more than one leg is connected. In particular, $H_{s-1}(\bar{X}_s)\cong H_{s-1}(\ti X_s)$ which is given by equation \eqref{eq:homrainbow}. Hence,
\begin{equation*}
H_{s-1}(\bar{X}_s;\mb Z)\cong\mathbb{Z}^\frac{(s-1)!}{2}.
\end{equation*}

Moreover, note that an analogous statement to Proposition \ref{prop:projhc} holds also for $\bar{X}_s$ or $\m {MRG}_{n,s}$. Here the projection $\pi: \m {MRG}_{n,s} \rightarrow \m {MG}_{n,s}$ forgets the coloring of edges and the type of vertices. 

For computer aided calculation of the rational homology of $\bar{X}_s$ the representation of a simplex has to be slightly modified to account for the new face relations induced by $\sim_*$. The color of shrunken edges is now relevant when identifying faces of different simplices. We therefore add a weight to each vertex that has more than one leg attached. This weight is the set of colors of the edges that were collapsed to that vertex (i.e.\ one for each additional leg). 

The homology groups in the one-loop case for up to five legs that were calculated with computer assistance are listed in Table \ref{t:hom_dirk2}. An explicit choice of generators can again be found in \cite{mmm}.
\begin{table}[h]
\centering
\begin{tabular}{ c || c | c | c | c | c }
   & $H_0$ & $H_1$ & $H_2$ & $H_3$ & $H_4$\\
	\hline\\[-1em]
$\bar{X}_1$ & 1 & - & - & - & - \\
$\bar{X}_2$ & 1 & 0 & - & - & - \\
$\bar{X}_3$ & 1 & 2 & 1 & - & - \\
$\bar{X}_4$ & 1 & 0 & 18 & 3 & - \\
$\bar{X}_5$ & 1 & 0 & 48 & 166 & 12 \\
\end{tabular}
\caption{The dimension of the homology groups $H_k(\bar{X}_s;\mb Q)$ for $1\leq s\leq5$ and $0\leq k\leq4$.}
\label{t:hom_dirk2}
\end{table}

\subsubsection{The Euler characteristic of $\bar{X}_s$}

We give an explicit formula for the Euler characteristic of $\bar{X}$ by explicitly counting the simplices in each dimension.

Let $k\geq1$ and consider the one-loop graph $G$ with $k$ vertices and no external legs. We start by fixing one coloring of the $k$ internal edges of this graph. There are $\binom{s}{k}$ choices of $k$ colors out of the total $s$ and for each such choice there are $\frac{1}{2}(1+\delta_{1,k})(k-1)!$ distinct ways to distribute these colors on the internal edges of $G$. Now we add the legs and the remaining colors. The number of ways to distribute the latter set depends on the partition of the former set induced by the vertices. Let $i_1,\ldots,i_k$ be the number of legs attached to the $k$ vertices, ordered by $1\leq i_1\leq\cdots\leq i_k$. There are $\binom{s-k}{i_1-1}$ choices for the colors at the vertex corresponding to $i_1$, $\binom{s-k-(i_1-1)}{i_2-1}$ for the one corresponding to $i_2$, and so on. Multiplying all these contributions yields a factor of $\frac{(s-k)!}{\prod_{n=1}^{k}(i_n-1)!}$. Furthermore, there are $\binom{s}{i_1}$ ways to label the legs at the vertex corresponding to $i_1$, $\binom{s-i_1}{i_2}$ for $i_2$, and so on. This leads to an additional factor of $\frac{s!}{\prod_{n=1}^{k}i_n!}$. Finally, we have to count all the ways to distribute one of these partitions of legs among the $k$ vertices. There are $k!$ ways to do this but since we already included all ways of labelling the legs, we need to consider all such ways that differ only by a permutation of vertices with equal valency as identical. Thus, with
\begin{align*}
g(i_1,\ldots,i_k):=\frac{k!}{\prod_{n\in\{i_1,\ldots,i_k\}}|\{m\in\{1,\ldots,k\}\;|\;i_m=n \}|!}
\end{align*}
we obtain the number of holocolored one-loop graphs with remembered edges with $k$ vertices and $s$ legs $N_{s,k}$ as
\begin{align*}
N_{s,k}&=(1+\delta_{1,k})\frac{(k-1)!}{2}\binom{s}{k}\sum_{\substack{1\leq i_1\leq\cdots\leq i_k \\ \sum_{n=1}^{k}i_n=s}} \frac{s!(s-k)!}{\prod_{n=1}^{k}i_n!(i_n-1)!}g(i_1,\ldots,i_k)\\
&=(1+\delta_{1,k})\frac{(s!)^2}{2k}\sum_{\substack{1\leq i_1\leq\cdots\leq i_k \\ \sum_{n=1}^{k}i_n=s}}\frac{g(i_1,\ldots,i_k)}{\prod_{n=1}^{k}i_n!(i_n-1)!}.
\end{align*}

The Euler characteristic $\chi(\bar{X}_s)$ can now readily be obtained by summation over all $k$ with alternating signs,
\begin{align*}
\chi(\bar{X}_s)=(s!)^2\sum_{k=1}^{s}\frac{(-1)^{k-1}(1+\delta_{1,k})}{2k}\sum_{\substack{1\leq i_1\leq\cdots\leq i_k \\ \sum_{n=1}^{k}i_n=s}}\frac{g(i_1,\ldots,i_k)}{\prod_{n=1}^{k}i_n!(i_n-1)!}.
\end{align*}

\bibliography{ref}
\bibliographystyle{alpha}

\end{document}